\documentclass[reqno]{article}

\usepackage[colorinlistoftodos]{todonotes}
\usepackage[utf8]{inputenc}
\usepackage{amsmath,amsfonts,amssymb,amsthm}
\usepackage{mathtools}
\usepackage{esint}

\usepackage{hyperref}
\hypersetup{
    colorlinks=true,
    linkcolor=blue,
    citecolor=red,
    urlcolor=blue,
    pdfborder={0 0 0}
}

\usepackage{cleveref}
  \crefname{section}{Section}{Sections}
  \crefname{figure}{Figure}{Figures}
  \crefname{theorem}{Theorem}{Theorems}
  \crefname{lemma}{Lemma}{Lemmas}
  \crefname{proposition}{Proposition}{Propositions}  
  \crefname{corollary}{Corollary}{Corollaries}
  \crefname{definition}{Definition}{Definitions}
  \crefname{example}{Example}{Examples}
  \crefname{remark}{Remark}{Remarks}

\newtheorem{theorem}{Theorem}[section]
\newtheorem{lemma}[theorem]{Lemma}
\newtheorem{proposition}[theorem]{Proposition}
\newtheorem{corollary}[theorem]{Corollary}

\theoremstyle{definition}

\theoremstyle{remark}
\newtheorem{remark}[theorem]{Remark}

\usepackage{etoolbox}

\DeclareFontFamily{U}{matha}{\hyphenchar\font45}
\DeclareFontShape{U}{matha}{m}{n}{
<-6> matha5 <6-7> matha6 <7-8> matha7
<8-9> matha8 <9-10> matha9
<10-12> matha10 <12-> matha12
}{}
\DeclareSymbolFont{matha}{U}{matha}{m}{n}

\DeclareFontFamily{U}{mathx}{\hyphenchar\font45}
\DeclareFontShape{U}{mathx}{m}{n}{
<-6> mathx5 <6-7> mathx6 <7-8> mathx7
<8-9> mathx8 <9-10> mathx9
<10-12> mathx10 <12-> mathx12
}{}
\DeclareSymbolFont{mathx}{U}{mathx}{m}{n}

\DeclareMathDelimiter{\vvvert} {0}{matha}{"7E}{mathx}{"17}%

\DeclarePairedDelimiterX{\normiii}[1]
{\vvvert}
{\vvvert}
{\ifblank{#1}{\:\cdot\:}{#1}}

\DeclarePairedDelimiter\abs{\lvert}{\rvert}
\DeclarePairedDelimiter\norm{\lVert}{\rVert}

\newcommand*{\defeq}{\mathrel{\mathop:}=}
\newcommand*{\eqdef}{=\mathrel{\mathop:}}

\renewcommand*{\Re}{\operatorname{Re}}
\renewcommand*{\Im}{\operatorname{Im}}

\newcommand*{\diam}{\operatorname{diam}}
\newcommand*{\dist}{\operatorname{dist}}

\newcommand*{\pr}{\mathbb P}
\newcommand*{\ex}{\mathbb E}

\newcommand*{\NN}{\mathbb N}

\newcommand*{\RR}{\mathbb R}
\newcommand*{\CC}{\mathbb C}

\newcommand*{\HH}{\mathbb H}
\newcommand*{\barH}{\overline\HH}

\newcommand*{\slek}{SLE$_\kappa$}

\title{Regularity of SLE in $(t,\kappa)$ and refined GRR estimates}
\author{Peter K. Friz\footnote{\texttt{friz@math.tu-berlin.de}} , 
Huy Tran\footnote{\texttt{tran@math.tu-berlin.de}} , 
Yizheng Yuan\footnote{\texttt{yuan@math.tu-berlin.de}} \\
TU and WIAS Berlin, TU Berlin, TU Berlin}

\begin{document}
\maketitle

%
%
%
%
%
%
%

\begin{abstract}
Schramm-Loewner evolution (SLE$_\kappa$) is classically studied via Loewner evolution with half-plane  capacity parametrization, driven by $\sqrt{\kappa}$ times Brownian motion. This yields a (half-plane) valued random field $\gamma = \gamma (t, \kappa; \omega)$. (H\"older) regularity of in $\gamma(\cdot,\kappa;\omega$), a.k.a. SLE trace, has been considered by many authors, starting with Rohde-Schramm (2005). Subsequently, Johansson Viklund, Rohde, and Wong (2014) showed a.s. H\"older continuity of this random field for $\kappa < 8(2-\sqrt{3})$. In this paper, we improve their result to joint H\"older continuity up to $\kappa < 8/3$. Moreover, we show that the \slek{} trace $\gamma(\cdot,\kappa)$ (as a continuous path) is stochastically continuous in $\kappa$ at all $\kappa \neq 8$. Our proofs rely on a novel variation of the Garsia-Rodemich-Rumsey (GRR) inequality, which is of independent interest. 
\end{abstract}

\tableofcontents

\section{Introduction}

Schramm-Loewner evolution (SLE) is a random (non-self-crossing) path connecting two boundary points of a domain. To be more precise, it is a family of such random paths indexed by a parameter $\kappa \ge 0$. It has been first introduced by O. Schramm (2000) to describe several random models from statistical physics. Since then, many authors have intensely studied this random object. Many connections to discrete processes and other geometric objects have been made, and nowadays SLE is one of the key objects in modern probability theory.

The typical way of constructing SLE is via the Loewner differential equation (see \cref{sec:sle_continuity}) which provides a correspondence between real-valued functions (``driving functions'') and certain growing families of sets (``hulls'') in a planar domain. For many (in particular more regular) driving functions, the growing families of hulls (or their boundaries) are continuous curves called traces. For Brownian motion, it is a non-trivial fact that for fixed $\kappa \ge 0$, the driving function $\sqrt{\kappa}B$ almost surely generates a continuous trace which we call \slek{} trace (see \cite{RS05,LSW04}).

There has been a series of papers investigating the analytic properties of SLE, such as (Hölder and $p$-variation) regularity of the trace \cite{RS05,Law09, JVL11, FT17}. See also \cite{FS17,STW19} for some recent attempts to understand better the existence of SLE trace. 

A natural question is whether the \slek{} trace obtained from this construction varies continuously in the parameter $\kappa$. Another natural question is whether with probability $1$ the construction produces a continuous trace simultaneously for all $\kappa \ge 0$. These questions have been studied in \cite{JVRW14} where the authors showed that with probability $1$, the \slek{} trace exists and is continuous in the range $\kappa \in [0,8(2-\sqrt{3})[$. In our paper we improve their result and extend it to $\kappa \in [0,8/3[$. (In fact, our result is a bit stronger than the following statement, see \cref{thm:gamma_existence_hoelder} and \cref{thm:fv_hoelder_convergence}.)

\begin{theorem}\label{thm:main}
Let $B$ be a standard Brownian motion. Then almost surely the \slek{} trace $\gamma^\kappa$ driven by $\sqrt{\kappa}B_t$, $t \in [0,1]$, exists for all $\kappa \in [0,8/3[$, and the trace (parametrised by half-plane capacity) is continuous in $\kappa \in [0,8/3[$ with respect to the supremum distance on $[0,1]$. 
\end{theorem}

Stability of SLE trace was also recently studied in \cite[Theorem 1.10]{KS17}. They show the law of $\gamma^{\kappa_n} \in C([0,1],\HH)$   converges weakly to the law of $\gamma^{\kappa}$ in the topology of uniform convergence, whenever $\kappa_n \to \kappa < 8$. Of course, we get this as a trivial corollary of Theorem \ref{thm:main} in case of $\kappa < 8/3$. Our \cref{thm:main2} (proved in \cref{sec:stochastic_continuity}) strengthens  \cite[Theorem 1.10]{KS17} in three ways: \\
(i) we allow for any $\kappa \neq 8$; \\
(ii) we improve weak convergence to convergence in probability; \\
(iii) we strengthen convergence in $C([0,1],\HH)$ with uniform topology to \linebreak$C^{p\text{-var}}([0,1],\HH)$ with optimal (cf. \cite{FT17}) $p$-variation parameter, i.e. any \linebreak$p > (1 + \kappa/8) \wedge 2$. The analogous statement for $\alpha$-Hölder topologies, $\alpha < \left(1-\frac{\kappa}{24+2\kappa-8\sqrt{8+\kappa}}\right) \wedge \frac{1}{2}$, is also true.

Here and below we write $\| f \|^p_{p\text{-var};[a,b]} \defeq \sup \sum_{[s,t]\in \pi} |f(t)-f(s)|^p$, with $\sup$ taken over all partitions $\pi$ of $[a,b]$. The following theorem will be proved as \cref{co:sle_stochastic_continuity}.

\begin{theorem} \label{thm:main2}
Let $B$ be a standard Brownian motion, and $\gamma^\kappa$ the \slek{} trace driven by $\sqrt{\kappa}B_t$, $t \in [0,1]$, (and parametrised by half-plane capacity). For any $\kappa > 0$, $\kappa \neq 8$ and any sequence $\kappa_n \to \kappa$ we then have $\|\gamma^\kappa-\gamma^{\kappa_n}\|_{p\text{-var};[0,1]} \to 0$ in probability, for any $p > (1 + \kappa / 8) \wedge 2$. 
\end{theorem}

There are two major new ingredients to our proofs. First, we prove in \cref{sec:f_diffkappa_moment_pf} a refined moment estimate for SLE increments in $\kappa$, improving upon \cite{JVRW14}.
Using standard notation \cite{RS05,Law05}, for $\kappa > 0$, we denote by $(g^\kappa_t)_{t \ge 0}$ the forward SLE flow driven by $\sqrt{\kappa}B$, $j=1,2$, and by $\hat f^\kappa_t = (g^\kappa_t)^{-1}( \cdot +\sqrt{\kappa}B_t)$ the recentred inverse flow, also defined in  \cref{sec:sle_continuity} below.

Write $a \lesssim b$ for $a \le Cb$, with suitable constant $C<\infty$. The improved estimate (\cref{thm:f_diffkappa_moment}) reads 
\begin{equation}\label{eq:f_diffkappa_moment}
    \ex |\hat f^\kappa_t(i\delta)-\hat f^{\tilde\kappa}_t(i\delta)|^p \lesssim |\sqrt{\kappa}-\sqrt{\tilde\kappa}|^p
\end{equation}
for $1 \le p < 1+\frac{8}{\kappa}$. The interest in this
estimate is when $p$ is close to $1 + 8/\kappa$. No such estimate can be extracted from \cite{JVRW14}, as we explain in some more detail in \cref{rm:grr_application_jvrw} below.

Secondly, our way of exploiting moment estimates such as \eqref{eq:f_diffkappa_moment} is fundamentally different in comparison with the Whitney-type partition technique of ``$(t,y,\kappa)$''-space \cite{JVRW14} (already seen in \cite{RS05} without $\kappa$), combined with a Borel-Cantelli argument. Our key tool here is a new higher-dimensional variant of the Garsia-Rodemich-Rumsey (GRR) inequality \cite{GRR70} which is useful in its own right, essentially whenever one deals with random fields with very ``different'' -- in our case $t$ and $\kappa$ --  variables. The GRR inequality has been a useful tool in stochastic analysis to pass from moment bounds for stochastic processes to almost sure estimates of their regularity.

Let us briefly discuss the existing
(higher-dimensional) GRR estimates (e.g. \cite[Exercise 2.4.1]{SV79}, \cite{AI96,FKP06,HL13}) and their shortcomings in our setting. When we try to apply one of these versions to SLE (as a two-parameter random field in $(t,\kappa)$), we wish to estimate moments of $|\gamma(t,\kappa)-\gamma(s,\tilde\kappa)|$, where we denote the \slek{} trace by $\gamma(\cdot,\kappa)$. In \cite{FT17}, the estimate
\[\ex|\gamma(t,\kappa)-\gamma(s,\kappa)|^\lambda \lesssim |t-s|^{(\lambda+\zeta)/2}\]
with suitable $\lambda>1$ and $\zeta$ has been given. We will show in \cref{thm:sle_diff_moments} that
\[\ex|\gamma(s,\kappa)-\gamma(s,\tilde\kappa)|^p \lesssim |\kappa-\tilde\kappa|^p\]
for suitable $p>1$. Applying this estimate with $p=\lambda$, we obtain an estimate for $\ex|\gamma(t,\kappa)-\gamma(s,\tilde\kappa)|^\lambda$, and can apply a GRR lemma from \cite{AI96} or \cite{FKP06}. The condition for applying it is $((\lambda+\zeta)/2)^{-1}+p^{-1} = ((\lambda+\zeta)/2)^{-1}+\lambda^{-1} < 1$. But in doing so, we do not use the best estimates available to us. That is, the above estimate typically holds for some $p > \lambda$. On the other hand, we can only estimate the $\lambda$-th moment (and no higher ones) of $|\gamma(t,\kappa)-\gamma(s,\kappa)|$. This asks for a version of the GRR lemma that respects distinct exponents in the available estimates, and is applicable when $((\lambda+\zeta)/2)^{-1}+p^{-1} < 1$ with $p > \lambda$ (a weaker condition than above).

We are going to prove the following refined GRR estimates in two dimensions, as required by our application, noting that extension to higher dimension follow the same argument.
\begin{lemma}
Let $G$ be a continuous function (defined on some rectangle) such that, for some integers $J_1, J_2$,
\begin{alignat*}{2}
|G(x_1,x_2)-G(y_1,y_2)| &\le |G(x_1,x_2)-G(y_1,x_2)| &&+ |G(y_1,x_2)-G(y_1,y_2)|\\
&\le \sum^{J_1}_{j=1} |A_{1j}(x_1,y_1;x_2)| &&+ \sum^{J_1}_{j=1} |A_{2j}(y_1;x_2,y_2)|.
\end{alignat*}
Suppose that for all $j$,
\begin{align*}
\iiint \frac{|A_{1j}(u_1,v_1;u_2)|^{q_{1j}}}{|u_1-v_1|^{\beta_{1j}}} \, du_1 \, dv_1 \, du_2 &< \infty,\\
\iiint \frac{|A_{2j}(v_1;u_2,v_2)|^{q_{2j}}}{|u_2-v_2|^{\beta_{2j}}} \, dv_1 \, du_2 \, dv_2 &< \infty.
\end{align*}
Then, under suitable conditions on the exponents,
\begin{equation*}
    |G(x_1,x_2)-G(y_1,y_2)| \lesssim |x_1-y_1|^{\gamma^{(1)}}+|x_2-y_2|^{\gamma^{(2)}}. \\
\end{equation*}
\end{lemma}

Observe that the exponents $q_{1j}, q_{2j}$ are allowed to vary, exactly as required for our application to SLE. We also note that the flexibility to have $J_1,J_2 >1$ is used in the proof of Theorem \ref{thm:main2} but not \ref{thm:main}.

One might ask whether one can further improve \cref{thm:main} to all $\kappa \ge 0$. With the methods of this paper, it would require a better moment estimate in the style of \eqref{eq:f_diffkappa_moment} with larger exponent on the right-hand side. If such an estimate were to hold true with arbitrarily large exponent on the right-hand side (and any suitable exponent on the left-hand side), which is not clear to us, almost sure continuity of the random field in all $(t,\kappa)$ with $\kappa \neq 8$ would follow.

\medskip
\textbf{Acknowledgements:} PKF and HT acknowledge funding from European Research Council through Consolidator Grant 683164. All authors would like to thank S. Rohde and A. Shekhar for stimulating discussions. Moreover, we thank the referees for their comments, in particular for pointing out the literature on metric entropy bounds and majorising measures, and for suggesting simplified arguments in the proofs of \cref{thm:grr} and \cref{thm:kolmogorov}.

%
%
%
%

\section{A Garsia-Rodemich-Rumsey lemma with mixed exponents}

In this section we prove a variant of the Garsia-Rodemich-Rumsey inequality and Kolmogorov's continuity theorem. The classical Kolmogorov's theorem goes by a ``chaining'' argument (see e.g. \cite[Theorem 1.4.1]{Kun90} or \cite[Appendix A.2]{Tal14}), but can also be obtained from the GRR inequality (see e.g. \cite[Corollary 2.1.5]{SV79}). In the case of proving Hölder continuity of processes, the GRR approach provides more powerful statements (cf. \cite[Appendix A]{FV10}). In particular, we obtain bounds on the Hölder constant of the process that are more informative and easier to manipulate, which will be useful in the proof of \cref{thm:fv_hoelder_convergence}. (Although there are drawbacks of the GRR approach when generalising to more refined modulus of continuity, see the discussion in \cite[Appendix A.4]{Tal14}.)

We discuss some of the extensive literature that deal with the generality of GRR and Kolmogorov's theorem. The reader may skip this discussion and continue straight with the results of this section.

There are some direct generalisations of GRR and Kolmogorov's theorem to higher dimensions, e.g. \cite[Exercise 2.4.1]{SV79}, \cite[Theorem 1.4.1]{Kun90}, \cite{AI96,FKP06,HL13}. Moreover, there have been more systematic studies in a general setting under the titles metric entropy bounds and majorising measures. They derive bounds and path continuity of stochastic processes mainly from the structure of certain pseudometrics that the processes induce on the parameter space, such as $d_X(s,t) \defeq (\ex\abs{X(s)-X(t)}^2)^{1/2}$. A large amount of the theory is found in the book by Talagrand \cite{Tal14}. These results due to, among others, R. M. Dudley, N. K\^{o}no, X. Fernique, M. Talagrand, and W. Bednorz. Their main purpose is to allow different stuctures of the parameter space and inhomogeneity of the stochastic process (see e.g. \cite{Kon80,Bed07,Tal14}).

We explain why the existing results do not cover the adaption that we are seeking in this section. The general idea for applying the theory of metric entropy bounds would be considering the metric $d_X(s,t) = (\ex\abs{X(s)-X(t)}^{q})^{1/q}$ for some $q>1$.

Let us consider a random process defined on the parameter space $T = [0,1]^2$ that satisfies
\begin{align}\label{eq:kolmogorov_demo}
\begin{split}
\ex\abs{X(s_1,s_2)-X(t_1,s_2)}^{q_1} &\le \abs{s_1-t_1}^{\alpha_1},\\
\ex\abs{X(t_1,s_2)-X(t_1,t_2)}^{q_2} &\le \abs{s_2-t_2}^{\alpha_2},
\end{split}
\end{align}
where $q_1$ and $q_2$ might be different, say $q_1 < q_2$. By Hölder's inequality,
\begin{equation}\label{eq:kolmogorov_demo_bad_hoelder}
\ex\abs{X(t_1,s_2)-X(t_1,t_2)}^{q_1} \le \left( \ex\abs{X(t_1,s_2)-X(t_1,t_2)}^{q_2} \right)^{q_1/q_2} .
\end{equation}
Write $t = (t_1,t_2)$, $s = (s_1,s_2)$. We may let
\[ (\ex\abs{X(s)-X(t)}^{q})^{1/q} \le \abs{s_1-t_1}^{\alpha_1/q_1}+\abs{s_2-t_2}^{\alpha_2/q_2} \eqdef \normiii{s-t} \eqdef d(s,t) \]
where we can take $q = q_1$ (but not $q=q_2$ without knowing any bounds on higher moments of $\abs{X(s_1,s_2)-X(t_1,s_2)}$).

We explain now that we have already lost some sharpness when we estimated \eqref{eq:kolmogorov_demo_bad_hoelder} using Hölder's inequality. Indeed, all the results \cite[Theorem 3]{Kon80}, \cite[(13.141)]{Tal14}, \cite[Theorem B.2.4]{Tal14}, \cite[Corollary 1]{Bed07} are based on finding an increasing convex function $\varphi$ such that
\begin{equation}\label{eq:metric_entropy_condition}
\ex\varphi\left(\frac{\abs{X(s)-X(t)}}{d(s,t)}\right) \le 1 .
\end{equation}
Observe that we can take $\varphi(x) = x^{q_1}$ at best. To apply any of these results, the condition turns out to be $\frac{1}{\alpha_1}+\frac{q_2}{q_1\alpha_2} < 1$. In fact, \cite[Theorem 13.5.8]{Tal14} implies that we cannot expect anything better just from the assumption \eqref{eq:metric_entropy_condition}. More precisely, the theorem states that in general, when we assume only  \eqref{eq:metric_entropy_condition}, in order to deduce any pathwise bounds for the process $X$, we need to have
\[ \int_0^\delta \varphi^{-1}\left(\frac{1}{\mu(B(t,\varepsilon))}\right) \, d\varepsilon < \infty , \]
with $B$ denoting the ball with respect to the metric $d$, and $\mu$ e.g. the Lebesgue measure. In our setup this turns out to the condition $\frac{1}{\alpha_1}+\frac{q_2}{q_1\alpha_2} < 1$.

We will show in \cref{thm:kolmogorov} that by using the condition \eqref{eq:kolmogorov_demo} instead of \eqref{eq:metric_entropy_condition}, we can relax this condition to $\frac{1}{\alpha_1}+\frac{1}{\alpha_2} < 1$. In case $\frac{1}{\alpha_1}+\frac{1}{\alpha_2} < 1 < \frac{1}{\alpha_1}+\frac{q_2}{q_1\alpha_2}$, this is an improvement. We have not found this possibility in any of the existing references.

\medskip

We now turn to our version of the Garsia-Rodemich-Rumsey inequality that allows us to make use of different exponents $q_1 \neq q_2$. In addition to the scenario \eqref{eq:kolmogorov_demo}, we allow also the situation when e.g. $\abs{X(s_1,s_2)-X(t_1,s_2)} \le A_{11}+A_{12}$ with $\ex\abs{A_{1j}}^{q_{1j}} \le \abs{s_1-t_1}^{\alpha_{1j}}$ for some $q_{1j},\alpha_{1j}$, $j=1,2$, where possibliy $q_{11} \neq q_{12}$.

Let $(E,d)$ be a metric space. We can assume $E$ to be isometrically embedded in some larger Banach space (by the Kuratowski embedding). To ease the notation, we write $\abs{x-y} = d(x,y)$ both for the distance in $E$ and for the distance in $\RR$. For a Borel set $A$ we denote by $\abs{A}$ its Lebesgue measure and $\fint_A f = \frac{1}{\abs{A}} \int_A f$.

In what follows, let $I_1$ and $I_2$ be two (either open or closed) non-trivial intervals of $\mathbb{R}$. 

\begin{lemma}\label{thm:grr}
Let $G \in C(I_1\times I_2)$ be a continuous function, with values in a metric space $E$, such that
\begin{equation}\label{e:G_assump}
|G(x_1,x_2)-G(y_1,y_2)
| \le \sum^{J_1}_{j=1} |A_{1j}(x_1,y_1;x_2)| + \sum^{J_2}_{j=1} |A_{2j}(y_1;x_2,y_2)|    
\end{equation}
for all $(x_1,x_2), (y_1,y_2) \in I_1 \times I_2$, where $A_{1j}:I_1 \times I_1 \times I_2\to \mathbb{R}$, $1\leq j\leq J_1$, $A_{2j}: I_1 \times I_2 \times I_2\to \mathbb{R}$, $1\leq j\leq J_2$, are measurable functions.  Suppose that 
\begin{align}
\iiint_{I_1 \times I_1 \times I_2} \frac{|A_{1j}(u_1,v_1;u_2)|^{q_{1j}}}{|u_1-v_1|^{\beta_{1j}}} \, du_1 \, dv_1 \, du_2 &\le M_{1j}, \label{eq:grr_condA}\\
\iiint_{I_1 \times I_2 \times I_2} \frac{|A_{2j}(v_1;u_2,v_2)|^{q_{2j}}}{|u_2-v_2|^{\beta_{2j}}} \, dv_1 \, du_2 \, dv_2 &\le M_{2j} \label{eq:grr_condB}
\end{align}
for all $j$, where $q_{ij} \ge 1$, $\beta_i := \min_j \beta_{ij} > 2$, $i=1,2$, and $(\beta_1-2)(\beta_2-2)-1 > 0$. Fix any $a,b > 0$. Then
\begin{multline}\label{e:G_diff}
|G(x_1,x_2)-G(y_1,y_2)| \le C \sum_j M_{1j}^{1/q_{1j}} \, \left(|x_1-y_1|^{\gamma^{(1)}_{1j}}+|x_2-y_2|^{\gamma^{(2)}_{1j}}\right) \\
+ C \sum_j M_{2j}^{1/q_{2j}} \, \left(|x_1-y_1|^{\gamma^{(1)}_{2j}}+|x_2-y_2|^{\gamma^{(2)}_{2j}}\right)
\end{multline}
for all $(x_1,x_2), (y_1,y_2) \in I_1 \times I_2$, where $\gamma^{(1)}_{1j} = \dfrac{\beta_{1j}-2-b}{q_{1j}}$, $\gamma^{(2)}_{1j} = \dfrac{(\beta_{1j}-2)a-1}{q_{1j}}$, $\gamma^{(1)}_{2j} = \dfrac{(\beta_{2j}-2)b-1}{q_{2j}}$, $\gamma^{(2)}_{2j} = \dfrac{\beta_{2j}-2-a}{q_{2j}}$, and $C < \infty$ is a constant that depends on $(q_{ij}),(\beta_{ij}),a,b,|I_1|,|I_2|$.
\end{lemma}

\begin{remark}
The statement is already true when $q_{ij}>0$ (not necessarily $\ge 1$) and can be shown by an argument similarly as in \cite[Theorem 2.1.3 and Exercise 2.4.1]{SV79}. We have decided to stick to $q_{ij} \ge 1$ since the proof is simpler here.
\end{remark}

\begin{proof}
Note that for any continuous function $G$ and a sequence $B_n$ of sets with $\diam(\{x\} \cup B_n) \to 0$ we have $G(x) = \lim_n \fint_{B_n} G$. (Recall that we can view $E$ as a subspace of some Banach space, so that the integral is well-defined.)

Let $(x_1,x_2), (y_1,y_2) \in I_1 \times I_2$. Using the above observation, we will approximate $G(x_1,x_2)$ and $G(y_1,y_2)$ by well-chosen sequences of sets.

We pick a sequence of rectangles $I^n_1 \times I^n_2 \subseteq I_1 \times I_2$, $n \ge 0$, with the following properties:
\begin{itemize}
    \item $(x_1,x_2), (y_1,y_2) \in I^0_1 \times I^0_2$.
    \item $(x_1,x_2) \in I^n_1 \times I^n_2$ for all $n$.
    \item $|I^n_i| = R_i^{-n} d_i$, $i=1,2$, with parameters \[ R_1,R_2 > 1,\quad d_1,d_2 > 0 \] chosen later.
\end{itemize}
In order for such a sequence of rectangles to exist, we must have \[\abs{x_i-y_i} \le d_i \le \abs{I_i}, \quad i=1,2, \] since we require $x_i,y_i \in I^0_i \subseteq I_i$. Conversely, this condition guarantees the existence of such a sequence. 

We will bound
\[ \abs*{G(x_1,x_2) - \fint\fint_{I^0_1 \times I^0_2} G} \le \sum_{n \in \NN} \abs*{\fint\fint_{I^n_1 \times I^n_2} G - \fint\fint_{I^{n-1}_1 \times I^{n-1}_2} G} . \]
The same argument applies also to $G(y_1,y_2)$ where we can pick the same initial rectangle $I^0_1 \times I^0_2$. Hence, this will give us a bound on $\abs{G(x_1,x_2)-G(y_1,y_2)}$.

By the assumption \eqref{e:G_assump} we have
\[ \begin{split}
&\abs*{\fint\fint_{I^n_1 \times I^n_2} G - \fint\fint_{I^{n-1}_1 \times I^{n-1}_2} G} \\
&\quad = \abs*{\fint\fint_{I^n_1 \times I^n_2}\fint\fint_{I^{n-1}_1 \times I^{n-1}_2} (G(u_1,u_2)-G(v_1,v_2)) \,du_1\,du_2\,dv_1\,dv_2}\\
&\quad \le \sum_j \fint_{I^n_1}\fint_{I^{n-1}_1}\fint_{I^n_2} \abs{A_{1j}(u_1,v_1;u_2)} + \sum_j\fint_{I^{n-1}_1}\fint_{I^n_2}\fint_{I^{n-1}_2}\abs{A_{2j}(v_1;u_2,v_2)} .
\end{split} \]

Recall that $\abs{I^n_i} = R_i^{-n}d_i$ and that $\abs{u_i-v_i} \le C R_i^{-n}d_i$ for any $u_i \in I^n_i$, $v_i \in I^{n-1}_i$. This and Hölder's inequality imply
\[ \begin{split}
&\fint_{I^n_1}\fint_{I^{n-1}_1}\fint_{I^n_2} \abs{A_{1j}(u_1,v_1;u_2)}\\
&\quad \le C (R_1^{-n}d_1)^{\beta_{1j}/q_{1j}} \fint_{I^n_1}\fint_{I^{n-1}_1}\fint_{I^n_2} \frac{\abs{A_{1j}(u_1,v_1;u_2)}}{\abs{u_1-v_1}^{\beta_{1j}/q_{1j}}}\\
&\quad \le C (R_1^{-n}d_1)^{\beta_{1j}/q_{1j}} \left( \fint_{I^n_1}\fint_{I^{n-1}_1}\fint_{I^n_2} \frac{\abs{A_{1j}(u_1,v_1;u_2)}^{q_{1j}}}{\abs{u_1-v_1}^{\beta_{1j}}} \right)^{1/q_{1j}}\\
&\quad \le C (R_1^{-n}d_1)^{\beta_{1j}/q_{1j}} \left( (R_1^{-n}d_1)^{-2} (R_2^{-n}d_2)^{-1} M_{1j} \right)^{1/q_{1j}} \\
&\quad = C \left( (R_1^{-n}d_1)^{\beta_{1j}-2} (R_2^{-n}d_2)^{-1} M_{1j} \right)^{1/q_{1j}} .
\end{split} \]
Similarly,
\[ \fint_{I^{n-1}_1}\fint_{I^n_2}\fint_{I^{n-1}_2}\abs{A_{2j}(v_1;u_2,v_2)} 
\le C \left( (R_2^{-n}d_2)^{\beta_{2j}-2} (R_1^{-n}d_1)^{-1} M_{2j} \right)^{1/q_{2j}} . \]

We want to sum the above expressions for all $n$, which is possible if and only if both $R_1^{\beta_{1j}-2} R_2^{-1} > 1$ and $R_2^{\beta_{2j}-2} R_1^{-1} > 1$. The best pick is $R_2 = R_1^{\frac{\beta_1-1}{\beta_2-1}}$ (the exact scale of $R_1$ does not matter), and the condition becomes $(\beta_1-2)(\beta_2-2)-1 > 0$ (assuming $\beta_1,\beta_2 > 2$). In that case, we finally get
\begin{multline} \label{eq:grr_summed}
\abs{G(x_1,x_2)-G(y_1,y_2)} \\
\le C \sum_j \left( d_1^{\beta_{1j}-2} d_2^{-1} M_{1j} \right)^{1/q_{1j}} + C \sum_j \left( d_2^{\beta_{2j}-2} d_1^{-1} M_{2j} \right)^{1/q_{2j}}
\end{multline}

It remains to pick $d_1,d_2 > 0$. Let $d_1 := |x_1-y_1| \vee |x_2-y_2|^a$, $d_2 := |x_1-y_1|^b \vee |x_2-y_2|$, and suppose for the moment that $d_1 \le |I_1|$, $d_2 \le |I_2|$. (The conditions $d_1 \ge |x_1 - y_1|$, $d_2 \ge |x_2 - y_2|$ are satisfied by our choice.). In this case the inequality \eqref{eq:grr_summed} becomes
\begin{equation} \begin{split}\label{eq:grr_final}
&\abs{G(x_1,x_2)-G(y_1,y_2)} \\
&\quad \le C \sum_j M_{1j}^{1/q_{1j}} \, \left(|x_1-y_1|^{\beta_{1j}-2-b}+|x_2-y_2|^{(\beta_{1j}-2)a-1}\right)^{1/q_{1j}}\\
&\qquad + C \sum_j M_{2j}^{1/q_{2j}} \, \left(|x_1-y_1|^{(\beta_{2j}-2)b-1}+|x_2-y_2|^{\beta_{2j}-2-a}\right)^{1/q_{2j}} .
\end{split} \end{equation}
This proves the claim in case $d_1 \le |I_1|$, $d_2 \le |I_2|$.

It remains to handle the case when $d_1 > |I_1|$ or $d_2 > |I_2|$. In that case we pick $\hat d_1 = d_1 \wedge |I_1|$ and $\hat d_2 = d_2 \wedge |I_2|$ instead of $d_1$ and $d_2$. The conditions $|x_1-y_1| \le \hat d_1 \le |I_1|$ and $|x_2-y_2| \le \hat d_2 \le |I_2|$ are now satisfied, and in \eqref{eq:grr_summed}, we instead have
\begin{align}\label{eq:dependence_intervals}
\begin{split}
    \hat d_1^{\beta_{1j}-2} \hat d_2^{-1} &\le \frac{d_2}{d_2 \wedge |I_2|} \ d_1^{\beta_{1j}-2} d_2^{-1} = \left( \frac{|x_1-y_1|^b}{|I_2|} \vee 1 \right) d_1^{\beta_{1j}-2} d_2^{-1},\\
    \hat d_1^{-1} \hat d_2^{\beta_{2j}-2} &\le \frac{d_1}{d_1 \wedge |I_1|} \ d_1^{-1} d_2^{\beta_{2j}-2} = \left( \frac{|x_2-y_2|^a}{|I_1|} \vee 1 \right) d_1^{-1} d_2^{\beta_{2j}-2},
\end{split}
\end{align}
i.e. the same result \eqref{eq:grr_final} holds with the additional constants $\left( \frac{|x_1-y_1|^b}{|I_2|} \vee 1 \right)$ and $\left( \frac{|x_2-y_2|^a}{|I_1|} \vee 1 \right)$ (which can be bounded by a constant depending on $a,b,|I_1|,|I_2|$ since $a,b \ge 0$).
\end{proof}

\begin{remark}\label{rm:grr_dependence_intervals}
The dependence of the multiplicative constant $C$ on $|I_1|$ and $|I_2|$ is specified in \eqref{eq:dependence_intervals}. This can be convenient when we want to apply the lemma to different domains.

A more accurate version is
\begin{align*}
    \hat d_1^{\beta_{1j}-2} \hat d_2^{-1} &= \left( \frac{d_1 \wedge |I_1|}{d_1} \right)^{\beta_{1j}-2} \frac{d_2}{d_2 \wedge |I_2|} \ d_1^{\beta_{1j}-2} d_2^{-1}\\
    &= \left( \frac{|I_1|}{|x_2-y_2|^a} \wedge 1 \right)^{\beta_{1j}-2} \left( \frac{|x_1-y_1|^b}{|I_2|} \vee 1 \right) d_1^{\beta_{1j}-2} d_2^{-1},\\
    \hat d_1^{-1} \hat d_2^{\beta_{2j}-2} &= \left( \frac{d_2 \wedge |I_2|}{d_2} \right)^{\beta_{2j}-2} \frac{d_1}{d_1 \wedge |I_1|} \ d_1^{-1} d_2^{\beta_{2j}-2}\\
    &= \left( \frac{|I_2|}{|x_1-y_1|^b} \wedge 1 \right)^{\beta_{2j}-2} \left( \frac{|x_2-y_2|^a}{|I_1|} \vee 1 \right) d_1^{-1} d_2^{\beta_{2j}-2}.
\end{align*}
\end{remark}

\begin{remark}
We could have added some more flexibility by allowing the exponents $(q_{ij}),(\beta_{ij})$ to vary with $u_1,u_2$, but again we will not need it for our result.
\end{remark}

\begin{remark}\label{rmk:grr_exponents}
We have a free choice of $a,b \ge 0$ which affects the Hölder exponents $\gamma^{(1)}_{ij}, \gamma^{(2)}_{ij}$. In general, it is not simple to spell out the optimal choice of $a,b$ and hence the optimal Hölder exponents. Usually we are interested in the overall exponents (i.e. $\min_{i,j} \gamma^{(1)}_{ij}$, $\min_{i,j} \gamma^{(2)}_{ij}$), and we can solve
\begin{align*}
\min_j \gamma^{(1)}_{1j} &= \min_j \gamma^{(1)}_{2j},\\
\min_j \gamma^{(2)}_{1j} &= \min_j \gamma^{(2)}_{2j}
\end{align*}
to find the optimal choice for $a,b$.

For instance, in case $\beta_{1j} = \beta_1$ and $\beta_{2j} = \beta_2$ for all $j$, the best choice is
\begin{align*}
a = \frac{q_1(\beta_2-2)+q_2}{q_2(\beta_1-2)+q_1},\quad b = \frac{q_2(\beta_1-2)+q_1}{q_1(\beta_2-2)+q_2},
\end{align*}
resulting in
\begin{align*}
\gamma^{(1)} = \frac{(\beta_1-2)(\beta_2-2)-1}{q_1(\beta_2-2)+q_2},\quad \gamma^{(2)} = \frac{(\beta_1-2)(\beta_2-2)-1}{q_2(\beta_1-2)+q_1}
\end{align*}
where $q_i = \max_j q_{ij}$.

In general, we could choose $a = \frac{\beta_2-1}{\beta_1-1}$, $b=\frac{\beta_1-1}{\beta_2-1}$, resulting in
\begin{alignat*}{2}
\gamma^{(1)}_{1j} &= \frac{(\beta_{1j}-2)(\beta_2-2)-1+\beta_{1j}-\beta_1}{q_{1j}(\beta_2-1)},\quad & \gamma^{(2)}_{1j} &= \frac{(\beta_{1j}-2)(\beta_2-2)-1+\beta_{1j}-\beta_1}{q_{1j}(\beta_1-1)},\\
\gamma^{(1)}_{2j} &= \frac{(\beta_1-2)(\beta_{2j}-2)-1+\beta_{2j}-\beta_2}{q_{2j}(\beta_2-1)},\quad & \gamma^{(2)}_{2j} &= \frac{(\beta_1-2)(\beta_{2j}-2)-1+\beta_{2j}-\beta_2}{q_{2j}(\beta_1-1)}.
\end{alignat*}
But this is not necessarily the optimal choice.
\end{remark}

\begin{remark}
Notice that the condition to apply the lemma does only depend on $(\beta_{ij})$, not $(q_{ij})$, but the resulting H\"older-exponents will.
\end{remark}

\begin{remark}
The proof straightforwardly generalises to higher dimensions.
\end{remark}


Using our version of the GRR lemma, we can show another version of the Kolmogorov continuity condition. Here we suppose $I_1$, $I_2$ are \textbf{bounded} intervals.

\begin{theorem}
\label{thm:kolmogorov}
Let $X$ be a random field on $I_1 \times I_2$ taking values in a separable Banach space. Suppose that, for $(x_1,x_2), (y_1,y_2) \in I_1 \times I_2$, we have
\begin{equation}\label{eq:kolmogorov_Xdiff_assumption}
|X(x_1,x_2)-X(y_1,y_2)| \le \sum_{j=1}^{J_1} |A_{1j}(x_1,y_1;x_2)| + \sum_{j=1}^{J_2} |A_{2j}(y_1;x_2,y_2)|
\end{equation}
with measurable real-valued $A_{ij}$ that satisfy
\begin{equation}\label{eq:kolmogorov_moment_assumption}
\begin{split}
\ex |A_{1j}(x_1,y_1;x_2)|^{q_{1j}} &\le C' \, |x_1-y_1|^{\alpha_{1j}},\\
\ex |A_{2j}(y_1;x_2,y_2)|^{q_{2j}} &\le C' \, |x_2-y_2|^{\alpha_{2j}}
\end{split}
\end{equation}
with a constant $C' < \infty$.

Moreover, suppose $q_{ij} \ge 1$, $\alpha_i = \min_j \alpha_{ij} > 1$, $i=1,2$, and $\alpha_1^{-1}+\alpha_2^{-1} < 1$.

Then $X$ has a H\"older-continuous modification $\hat X$. Moreover, for any
\begin{align*}
\gamma^{(1)} < \frac{(\alpha_1-1)(\alpha_2-1)-1}{q_1(\alpha_2-1)+q_2},\quad \gamma^{(2)} < \frac{(\alpha_1-1)(\alpha_2-1)-1}{q_2(\alpha_1-1)+q_1},
\end{align*}
where $q_i = \max_j q_{ij}$, there is a random variable $C$ such that
\begin{align*}
|\hat X(x_1,x_2)-\hat X(y_1,y_2)| \le C \left(|x_1-y_1|^{\gamma^{(1)}}+|x_2-y_2|^{\gamma^{(2)}}\right)
\end{align*}
and $\ex[C^{q_{\text{min}}}] < \infty$ for $q_{\text{min}}=\min_{i,j} q_{ij}$.
\end{theorem}

\begin{remark}
In case $\alpha_{1j} = \alpha_1$ and $\alpha_{2j} = \alpha_2$ for all $j$, the expressions for the Hölder exponents $\gamma^{(1)}, \gamma^{(2)}$ given above are sharp. In the general case, the exponents may be improved, following an
optimisation described in \cref{rmk:grr_exponents}.
\end{remark}

\begin{remark}
The constants $C'$ can be replaced by (deterministic) functions that are integrable in $(x_1, x_2)$, without change of the proof. But one would need to formulate the condition more carefully, therefore we decided to not include it.
\end{remark}

We point out that in case $J_1=J_2=1$ and $q_1=q_2$, this agrees with the two-dimensional version of the (inhomogeneous) Kolmogorov criterion \cite[Theorem 1.4.1]{Kun90}.

\begin{proof}\textbf{Part 1.} Suppose first that $X$ is already continuous. In that case we can directly apply \cref{thm:grr}. The expectation of the integrals \eqref{eq:grr_condA} and \eqref{eq:grr_condB} are finite if $\beta_{ij} < \alpha_{ij}+1$ for all $i,j$. By choosing $\beta_{ij}$ as large as possible, the conditions $(\beta_1-2)(\beta_2-2)-1 > 0$ and $\beta_1 > 2$, $\beta_2 > 2$ are satisfied if $\alpha_1^{-1}+\alpha_2^{-1} < 1$ and $\alpha_1 > 1$, $\alpha_2 > 1$. 

Since the (random) constants $M_{ij}$ in \cref{thm:grr} are almost surely finite, $X$ is H\"older continuous as quantified in \eqref{e:G_diff}, and the H\"older constants $M_{ij}^{1/q_{ij}}$ have $q_{ij}$-th moments since they are just the integrals \eqref{eq:grr_condA}. The formulas for the H\"older exponents follow from the analysis in \cref{rmk:grr_exponents}.

\textbf{Part 2.} Now, suppose $X$ is arbitrary. We need to construct a continuous version of $X$. It suffices to show that $X$ is uniformly continuous on a dense set $D \subseteq I_1 \times I_2$. Indeed, we can then apply Doob's separability theorem to obtain a separable (and hence continuous) version of $X$, or alternatively construct $\hat X$ by setting $\hat X = X$ on $D$ and extend $\hat X$ continuously to $I_1 \times I_2$. Then $\hat X$ is a modification of $X$ because they agree on a dense set $D$ and are both stochastically continuous (as follows from \eqref{eq:kolmogorov_Xdiff_assumption} and \eqref{eq:kolmogorov_moment_assumption}).

We use a standard argument that can be found e.g. in \cite[p. 8--9]{Tal90}.

We can assume without loss of generality that $X(\bar x_1,\bar x_2) = 0$ for some $(\bar x_1,\bar x_2) \in I_1 \times I_2$ (otherwise just consider $Y(x_1,x_2) = X(x_1,x_2)-X(\bar x_1,\bar x_2)$).

In particular, the conditions \eqref{eq:kolmogorov_Xdiff_assumption} and \eqref{eq:kolmogorov_moment_assumption} imply that $X(x_1,x_2)$ is an integrable random variable with values in a separable Banach space for every $(x_1,x_2)$.

Fix any countable dense subset $D \subseteq I_1 \times I_2$. Let
\[ \mathcal G \defeq \sigma( \{X(x_1,x_2) \mid (x_1,x_2) \in D\} ) . \]
We can pick an increasing sequence of \textbf{finite} $\sigma$-algebras $\mathcal G_n$ such that $\mathcal G = \sigma\left(\bigcup_n \mathcal G_n\right)$. By martingale convergence, we have
\[ X^{(n)}(x_1,x_2) \to X(x_1,x_2) \]
almost surely for $(x_1,x_2) \in D$ where $X^{(n)}(x_1,x_2) \defeq \ex[X(x_1,x_2) \mid \mathcal G_n]$.

Moreover, \eqref{eq:kolmogorov_Xdiff_assumption} implies
\[ |X^{(n)}(x_1,x_2)-X^{(n)}(y_1,y_2)| \le \sum_{j=1}^{J_1} |A_{1j}^{(n)}(x_1,y_1;x_2)| + \sum_{j=1}^{J_2} |A_{2j}^{(n)}(y_1;x_2,y_2)| \]
where $\abs{A_{ij}^{(n)}(...)} \defeq \ex[\abs{A_{ij}^{(n)}(...)} \mid \mathcal G_n]$. By Jensen's inequality and \eqref{eq:kolmogorov_moment_assumption}, we have
\[ \begin{split}
\ex |A_{1j}^{(n)}(x_1,y_1;x_2)|^{q_{1j}} &\le \ex |A_{1j}(x_1,y_1;x_2)|^{q_{1j}} \le C' \, |x_1-y_1|^{\alpha_{1j}},\\
\ex |A_{2j}^{(n)}(y_1;x_2,y_2)|^{q_{2j}} &\le \ex |A_{2j}(y_1;x_2,y_2)|^{q_{2j}} \le C' \, |x_2-y_2|^{\alpha_{2j}}.
\end{split} \]

In particular, $X^{(n)}$ is stochastically continuous, and since $\mathcal G_n$ is finite, $X^{(n)}$ is almost surely continuous. Applying \cref{thm:grr} yields
\begin{multline*}
|X^{(n)}(x_1,x_2)-X^{(n)}(y_1,y_2)| \le C \sum_j (M^{(n)}_{1j})^{1/q_{1j}} \, \left(|x_1-y_1|^{\gamma^{(1)}_{1j}}+|x_2-y_2|^{\gamma^{(2)}_{1j}}\right) \\
+ C \sum_j (M^{(n)}_{2j})^{1/q_{2j}} \, \left(|x_1-y_1|^{\gamma^{(1)}_{2j}}+|x_2-y_2|^{\gamma^{(2)}_{2j}}\right)
\end{multline*}
where $M^{(n)}_{ij}$ are defined as the integrals \eqref{eq:grr_condA} and \eqref{eq:grr_condB} with $A_{ij}^{(n)}$.

It follows that on $D$ we have
\begin{multline*}
|X(x_1,x_2)-X(y_1,y_2)| \le C \sum_j \tilde M_{1j}^{1/q_{1j}} \, \left(|x_1-y_1|^{\gamma^{(1)}_{1j}}+|x_2-y_2|^{\gamma^{(2)}_{1j}}\right) \\
+ C \sum_j \tilde M_{2j}^{1/q_{2j}} \, \left(|x_1-y_1|^{\gamma^{(1)}_{2j}}+|x_2-y_2|^{\gamma^{(2)}_{2j}}\right)
\end{multline*}
where $\tilde M_{ij} \defeq \liminf_n M^{(n)}_{ij}$. By Fatou's lemma,
\[ \ex \tilde M_{ij} \le \liminf_n \ex M^{(n)}_{ij} < \infty , \]
implying that $\tilde M_{ij} < \infty$, hence $X$ is uniformly continuous on $D$.
\end{proof}

One-dimensional variants of \cref{thm:grr} and \cref{thm:kolmogorov} can also be derived. Having shown the two-dimensional results \cref{thm:grr} and \cref{thm:kolmogorov}, there is no need for an additional proof of their one-dimensional variants, since we can extend any one-parameter function $G$ to a two-parameter function via $\tilde G(x_1,x_2) := G(x_1)$. This immediately implies the following results.

\begin{corollary}\label{thm:grr_1d}
Let $G$ be a continuous function on an interval $I$ such that
\begin{equation*}
    |G(x)-G(y)| \le \sum_{j=1}^J |A_j(x,y)|
\end{equation*}
for all $x,y \in I$, where $A_j: I \times I \to \RR$, $j=1,...,J$, are measurable functions that satisfy
\begin{equation*}
    \iint_{I \times I} \frac{|A_j(u,v)|^{q_j}}{|u-v|^{\beta_j}} \, du \, dv \le M_j
\end{equation*}
with some $q_j \ge 1$, $\beta_j > 2$. Then 
\begin{equation*}
    |G(x)-G(y)| \le C \sum_j M_j^{1/q_j} |x-y|^{\gamma_j}
\end{equation*}
for all $x,y \in I$, where $\gamma_j = \frac{\beta_j-2}{q_j}$, and $C < \infty$ is a constant that depends on $(q_j),(\beta_j)$.
\end{corollary}

For the sake of completeness we also state the one-dimensional version of Theorem \ref{thm:kolmogorov}.

\begin{corollary}
Let $X$ be a stochastic process on a bounded interval $I$ such that
\begin{equation*}
    |X(x)-X(y)| \le \sum_{j=1}^J |A_j(x,y)|
\end{equation*}
for all $x,y \in I$, where $A_j$, $j=1,...,J$, are measurable and satisfy
\begin{equation*}
    \ex |A_j(x,y)|^{q_j} \le C' |x-y|^{\alpha_j}
\end{equation*}
with $q_j \ge 1$, $\alpha_j > 1$, and $C' < \infty$.

Then $X$ has a continuous modification $\hat X$ that satisfies, for any $\gamma < \min_j \frac{\alpha_j-1}{q_j}$,
\begin{equation*}
    |\hat X(x)-\hat X(y)| \le C_\gamma |x-y|^{\gamma}
\end{equation*}
with a random variable $C_\gamma$ with $\ex[C_\gamma^{q_{\text{min}}}] < \infty$ where $q_{\text{min}} = \min_j q_j$.
\end{corollary}


\subsection{Further variations on the GRR theme}

We give some additional results that are similar or come as consequence of \cref{thm:grr}. This demonstrates the flexibility and generality that our lemma provides. We do not aim for a complete survey of all implications of the lemma.

We begin by proving the result of \cref{thm:grr} under slightly weaker assumptions. The assumptions may seem a bit at random, but they will turn out to be what we need in the proof of \cref{thm:fv_hoelder_convergence}.

\begin{lemma}\label{le:grr_relaxed}
Consider the same conditions as in \cref{thm:grr}, but instead of \eqref{e:G_assump}, we assume the following weaker condition. Let $r_j > 1$ and $\theta_j > 0$ such that $\frac{\beta_{1j}-2}{q_{1j}} < \theta_j$ for $j=1,...,J_1$.\footnote{A slightly different result still holds if $\frac{\beta_{1j}-2}{q_{1j}} \ge \theta_j$, as one can see in the proof.}
Suppose that for some small $c > 0$, e.g. $c \le \abs{I_1}/4$, we have
\begin{equation} \begin{split}\label{e:G_assump_weaker}
&\abs{G(x_1,x_2)-G(y_1,y_2)} \\
&\quad \le \sum^{J_1}_{j=1} \sum_{k=0}^{\lfloor \log_{r_j}(c/\abs{x_1-y_1}) \rfloor} r_j^{-k\theta_j} |A_{1j}(z_1+r_j^k(x_1-z_1),z_1+r_j^k(y_1-z_1);x_2)| \\
&\qquad + \sum^{J_2}_{j=1} |A_{2j}(y_1;x_2,y_2)|   
\end{split} \end{equation}
for $(x_1,x_2),(y_1,y_2) \in I_1 \times I_2$ and $z_1 \in I_1$ whenever $\abs{x_1-z_1} \vee \abs{y_1-z_1} \le 2\abs{x_1-y_1}$ and all the points appearing in the sum are also in the domain $I_1$.

Then the result of \cref{thm:grr} still holds, with the constant $C$ depending also on $(r_j),(\theta_j)$.
\end{lemma}

\begin{proof}
We proceed similarly as in the proof of \cref{thm:grr}. We pick the sequence $I^n_i$ a bit more carefully. Let $d_i > 0$, $R_i > 1$, $i=1,2$, be as in the proof of \cref{thm:grr}, and recall that we can freely pick $R_i \ge 9$. It is not hard to see that we can then pick a sequence of rectangles $I^n_1 \times I^n_2$ in such a way that
\begin{itemize}
\item $\abs{I^n_i} = \frac{1}{9} R_i^{-n}d_i$,
\item $\frac{1}{9} R_i^{-n}d_i \le \dist(I^n_i,I^{n+1}_i) \le R_i^{-n}d_i$,
\item $\dist(x_i,I^n_i) \to 0$ as $n \to \infty$,
\end{itemize}
and another analogous sequence of rectangles for $(y_1,y_2)$ that begins with the same $I^0_1 \times I^0_2$.

The proof proceeds in the same way, but instead of the assumption \eqref{e:G_assump}, we apply \eqref{e:G_assump_weaker} with some $z_1$ that we pick now.

Let $n \in \NN$. We pick $z_1 \defeq \inf(I^n_1 \cup I^{n-1}_1)$ if this point is in the left half of $I_1$, and $z_1 = \sup(I^n_1 \cup I^{n-1}_1)$ otherwise. From the defining properties of the sequence $(I^n_1)$ it follows that $\abs{u_1-z_1} \vee \abs{v_1-z_1} \le 2\abs{u_1-v_1}$ for all $u_1 \in I^n_1$, $v_1 \in I^{n-1}_1$. Moreover, all the points $z_1+r^k(u_1-z_1)$ and $z_1+r^k(v_1-z_1)$, $k \le \lfloor \log_r(c/\abs{x_1-y_1}) \rfloor$, are inside $I_1$ because $\abs{r^k(u_1-z_1)} \le \frac{c}{\abs{u_1-v_1}}\abs{u_1-z_1} \le 2c$ and we have chosen $z_1$ to be more than distance $\abs{I_1}/2 \ge 2c$ away (in the $u_1$ resp. $v_1$ direction) from the end of the interval $I_1$.

We now have to bound
\[ 
\sum_k \fint_{I^n_1}\fint_{I^{n-1}_1}\fint_{I^n_2} r^{-k\theta_j} \abs{A_{1j}(z_1+r^k(u_1-z_1),z_1+r^k(v_1-z_1);u_2)} \,du_2\,dv_1\,du_1
\]

With the transformation $\phi_k(u_1) = z_1+r^k(u_1-z_1)$ we get
\[ \begin{split}
&\fint_{I^n_1}\fint_{I^{n-1}_1}\fint_{I^n_2} r^{-k\theta_j} \abs{A_{1j}(z_1+r^k(u_1-z_1),z_1+r^k(v_1-z_1);u_2)} \\
&\quad = r^{-k\theta_j} \fint_{\phi_k(I^n_1)}\fint_{\phi_k(I^{n-1}_1)}\fint_{I^n_2} \abs{A_{1j}(u_1,v_1;u_2)} \\
&\quad \le C r^{-k\theta_j} (r^k R_1^{-n} d_1)^{\beta_{1j}/q_{1j}} \fint_{\phi_k(I^n_1)}\fint_{\phi_k(I^{n-1}_1)}\fint_{I^n_2} \frac{\abs{A_{1j}(u_1,v_1;u_2)}}{\abs{u_1-v_1}^{\beta_{1j}/q_{1j}}}\\
&\quad \le C r^{-k\theta_j} (r^k R_1^{-n} d_1)^{\beta_{1j}/q_{1j}} \left( \fint_{\phi_k(I^n_1)}\fint_{\phi_k(I^{n-1}_1)}\fint_{I^n_2} \frac{\abs{A_{1j}(u_1,v_1;u_2)}^{q_{1j}}}{\abs{u_1-v_1}^{\beta_{1j}}} \right)^{1/q_{1j}} \\
&\quad \le C r^{-k\theta_j} (r^k R_1^{-n} d_1)^{\beta_{1j}/q_{1j}} \left( (r^k R_1^{-n} d_1)^{-2} (R_2^{-n} d_2)^{-1} M_{1j} \right)^{1/q_{1j}} \\
&\quad = C r^{k((\beta_{1j}-2)/q_{1j}-\theta_j)} \left( (R_1^{-n} d_1)^{\beta_{1j}-2} (R_2^{-n} d_2)^{-1} M_{1j} \right)^{1/q_{1j}} .
\end{split} \]
Since we assumed $\frac{\beta_{1j}-2}{q_{1j}} < \theta_j$ this bound sums in $k$ to
\[ C \left( (R_1^{-n} d_1)^{\beta_{1j}-2} (R_2^{-n} d_2)^{-1} M_{1j} \right)^{1/q_{1j}} \]
which is the same bound as in the proof of \cref{thm:grr}. The rest of the proof is the same as in \cref{thm:grr}.
\end{proof}

The following corollary is only used for Theorem \ref{thm:sle_pvar_kappa}.

\begin{corollary}\label{thm:grr_pvar}
Consider the same conditions as in \cref{thm:grr}. For $x_1 \in I_1$, consider $G(x_1,\cdot)$ as an element in the space of continuous functions $C^0(I_2)$. Then the $p$-variation of $x_1 \mapsto G(x_1,\cdot)$ is at most
\begin{equation*}
    C \sum_j M_{1j}^{1/q_{1j}} |I_1|^{\gamma^{(1)}_{1j}} + C \sum_j M_{2j}^{1/q_{2j}} |I_1|^{\gamma^{(1)}_{2j}},
\end{equation*}
where $p = \max_{i,j} \frac{q_{ij}}{1+\gamma^{(1)}_{ij} q_{ij}} = \max_j \frac{q_{1j}}{\beta_{1j}-1-b} \vee \max_j \frac{q_{2j}}{(\beta_{2j}-2)b}$ (with a choice of $b \ge 0$), and $C$ does not depend on $|I_1|$.
\end{corollary}

\begin{proof}
Let $t^0 < t^1 < ... < t^n$ be a partition of $I_1$. The $p$-variation of $x_1 \mapsto G(x_1,\cdot) \in C^0(I_2)$ is
\begin{equation*}
    \sup_{\text{partitions of }I_1} \left( \sum_k \sup_{x_2 \in I_2} |G(t^k,x_2)-G(t^{k-1},x_2)|^p \right)^{1/p}.
\end{equation*}
We estimate the differences using \cref{thm:grr}, applied to $[t^{k-1},t^k] \times I_2$. Observe that since consider the difference only in the first parameter of $G$, the constant $C$ in the statement of \cref{thm:grr} does not depend on the size of $[t^{k-1},t^k]$, as we explained in \cref{rm:grr_dependence_intervals}. Hence we have
\begin{multline*}
|G(t^k,x_2)-G(t^{k-1},x_2)| \le C \sum_j \left( M_{1j}\big|_{[t^{k-1},t^k]} \right)^{1/q_{1j}} |t^k-t^{k-1}|^{\gamma^{(1)}_{1j}}\\
+ C \sum_j \left( M_{2j}\big|_{[t^{k-1},t^k]} \right)^{1/q_{2j}} |t^k-t^{k-1}|^{\gamma^{(1)}_{2j}}
\end{multline*}
for all $x_2 \in I_2$, where we denote by $M_{1j}\big|_{[s,t]}$ and $M_{2j}\big|_{[s,t]}$ the integrals in \eqref{eq:grr_condA} and \eqref{eq:grr_condB} restricted to $[s,t] \times [s,t] \times I_2$ and $[s,t] \times I_2 \times I_2$, respectively.

Similary to \cite[Corollary A.3]{FV10}, we can show that
\begin{equation*}
    \omega(s,t) = C^p \sum_j \left( M_{1j}\big|_{[s,t]} \right)^{p/q_{1j}} |s-t|^{p\gamma^{(1)}_{1j}}+ C^p \sum_j \left( M_{2j}\big|_{[s,t]} \right)^{p/q_{2j}} |s-t|^{p\gamma^{(1)}_{2j}}
\end{equation*}
is a control.
\end{proof}


\section{Continuity of SLE in $\kappa$ and $t$}
\label{sec:sle_continuity}

In this section we show the main results \cref{thm:main,thm:main2}. We adopt notations and prerequisite from \cite{JVRW14}. For the convenience of the reader, we quickly recall some important notations.

Let $U\colon [0,1]\to \mathbb{R}$ be continuous. The Loewner differential equation is the following initial value ODE
\begin{equation}\label{eq:loewner}
    \partial_t g_t(z) = \frac{2}{g_t(z)-U(t)},\quad g_0(z) = z \in \mathbb{H}.
\end{equation}
For each $z\in \mathbb{H}$, the ODE has a unique solution up to a time $T_z=\sup\{t>0:|g_t(z)-U(t)|>0\} \in (0,\infty]$. For $t \ge 0$, let $H_t = \{z\in \mathbb{H}:T_z>t\}$. It is known that $g_t$ is a conformal map from $H_t$ onto $\mathbb{H}.$ Define $f_t = g_t^{-1}$ and $\hat{f}_t = f_t( \cdot + U(t))$. One says that $\lambda$ generates a curve $\gamma$ if 
\begin{equation}\label{eq:trace_limit}
    \gamma(t):=\lim_{y\to 0^+} f_t(iy+U(t))
\end{equation}
exists and is continuous in $t\in [0,1]$. This is equivalent to saying that there exists a continuous $\barH$-valued path $\gamma$ such that for each $t \in [0,1]$, the domain $H_t$ is the unbounded connected component of $\HH \setminus \gamma[0,t]$.

It is known (\cite{RS05,LSW04}) that for fixed $\kappa\in [0,\infty)$, the driving function $U=\sqrt{\kappa}B$, where $B$ is a standard Brownian motion, almost surely generates a curve, which we will denote by $\gamma(\cdot,\kappa)$ or $\gamma^\kappa$. But we do not know whether given a Brownian motion $B$, almost surely all driving functions $\sqrt{\kappa}B$, $\kappa \ge 0$, simultaneously generate a curve. Furthermore, simulations suggest that for a fixed sample of $B$, the curve $\gamma^\kappa$ changes continuously in $\kappa$, but only partial proofs have been found so far. We remark that this question is not trivial to answer because in general, the trace does not depend continuously on its driver, as \cite[Example 4.49]{Law05} shows.

In \cite{JVRW14} the authors show that in the range $\kappa \in {[0, 8(2-\sqrt 3)[} \approx {[0, 2.1[}$, the answer to both of the above questions is positive. Our result \cref{thm:gamma_existence_hoelder} improves the range to $\kappa \in {[0, 8/3[}$.

%
%
%

We will often use the following bounds for the moments of $|\hat f_t'(iy)|$ that have been shown by F. Johansson-Viklund and G. Lawler in \cite{JVL11}. In order to state them, we use the following notation. Let $\kappa \ge 0$. Set
\begin{equation}\label{eq:moment_estimate_parameters}
\begin{split}
r_c = r_c(\kappa) &:= \frac{1}{2}+\frac{4}{\kappa},\\
\lambda(r) = \lambda(\kappa,r) &:= r\left(1+\frac{\kappa}{4}\right)-\frac{\kappa r^2}{8},\\
\zeta(r) = \zeta(\kappa,r) &:= r-\frac{\kappa r^2}{8}
\end{split}
\end{equation}
for $r<r_c(\kappa)$.

With the scaling invariance of SLE, \cite[Lemma 4.1]{JVL11} implies the following.
\begin{lemma}[{\cite[Lemma 2.1]{FT17}\footnote{Note that in \cite{FT17}, $\lambda$ was called $q$.}}]\label{thm:moment_estimate}
Let $\kappa > 0$, $r < r_c(\kappa)$. There exists a constant $C<\infty$ depending only on $\kappa$ and $r$ such that for all $t,y \in {]0,1]}$
\begin{align*}
\ex[|\hat f_t'(iy)|^{\lambda(r)}] \le C a(t) y^{\zeta(r)}
\end{align*}
where $a(t) = a(t,\zeta(r)) = t^{-\zeta(r)/2} \vee 1$.

Moreover, $C$ can be chosen independently of $\kappa$ and $r$ when $\kappa$ is bounded away from $0$ and $\infty$, and $r$ is bounded away from $-\infty$ and $r_c(\kappa)$.\footnote{Note that in \cite{JVL11}, the notation $a=2/\kappa$ and $q=r_c-r$ is used.}
\end{lemma}

Now, for a standard Brownian motion $B$, and an \slek{} flow driven by $\sqrt{\kappa}B$, we write $\hat f^\kappa_t$, $\gamma^\kappa$, etc.

We also use the following notation from \cite{JVL11}.
\begin{equation*}
v(t,\kappa,y) := \int_0^y |(\hat f^\kappa_t)'(iu)| \, du.
\end{equation*}
Observe that $v(t,\kappa,\cdot)$ is decreasing in $y$ and
\begin{equation*}
|\hat f^\kappa_t(iy_1)-\hat f^\kappa_t(iy_2)| \le \int_{y_1}^{y_2} |(\hat f^\kappa_t)'(iu)| \, du = |v(t,\kappa,y_1)-v(t,\kappa,y_2)|.
\end{equation*}
Therefore $\lim_{y\searrow 0}\hat f^\kappa_t(iy)$ exists if $v(t,\kappa,y)<\infty$ for some $y>0$. For fixed $t$, $\kappa$, this happens almost surely because \cref{thm:moment_estimate} implies
\begin{equation*}
    \ex v(t,\kappa,y) = \int_0^y \ex|(\hat f^\kappa_t)'(iu)| \, du < \infty.
\end{equation*}
So we can define
\begin{equation*}
    \gamma(t,\kappa) = \begin{cases}
    \lim_{y \searrow 0} \hat f^\kappa_t(iy) & \text{if the limit exists,}\\
    \infty & \text{otherwise,}
    \end{cases}
\end{equation*}
as a random variable. Note that with this definition we can still estimate
\begin{equation*}
    |\gamma(t,\kappa)-\hat f^\kappa_t(iy)| \le v(t,\kappa,y).
\end{equation*}

\subsection{Almost sure regularity of SLE in $(t,\kappa)$}

In this subsection, we prove our first main result.

\begin{theorem}\label{thm:gamma_existence_hoelder}
Let $0 < \kappa_- < \kappa_+ < 8/3$. Let $B$ be a standard Brownian motion. Then almost surely the \slek{} trace $\gamma^\kappa$ driven by $\sqrt{\kappa}B$ exists for all $\kappa \in [\kappa_-,\kappa_+]$. Moreover, there exists a random variable $C$, depending on $\kappa_-$, $\kappa_+$, such that
\begin{align*}
|\gamma(t,\kappa)-\gamma(s,\tilde\kappa)| \le C (|t-s|^\alpha + |\kappa-\tilde\kappa|^\eta)
\end{align*}
for all $t,s \in [0,1]$, $\kappa,\tilde\kappa \in [\kappa_-,\kappa_+]$ where $\alpha, \eta > 0$ depend on $\kappa_+$. Moreover, $C$ can be chosen to have finite $\lambda$th moment for some $\lambda>1$.
\end{theorem}

The theorem should be still true near $\kappa \approx 0$ (Without any integrability statement for $C$, it is shown in \cite{JVRW14}.), but due to complications in applying \cref{thm:moment_estimate} (cf. \cite[Proof of Lemma 3.3]{JVRW14}), we decided to omit it.

As in \cite{FT17}, we will estimate moments of the increments of $\gamma$, using \cref{thm:moment_estimate}. We need to be a little careful, though, when applying \cref{thm:moment_estimate}, that the exponents do depend on $\kappa$. Since we are going to apply that estimate a lot, let us agree on the following.

For every $\kappa > 0$, we will choose some $r_\kappa < r_c(\kappa)$, and we will call $\lambda_\kappa = \lambda(\kappa, r_\kappa)$ and $\zeta_\kappa = \zeta(\kappa, r_\kappa)$ (where $r_c$, $\lambda$, and $\zeta$ are defined in \eqref{eq:moment_estimate_parameters}). (The exact choices of $r_\kappa$ will be decided later.)

We will use the following moment estimates.
\begin{proposition}\label{thm:sle_diff_moments}
Let $0 < \kappa_- < \kappa_+ < \infty$. Let $t,s \in [0,1]$, $\kappa, \tilde\kappa \in [\kappa_-, \kappa_+]$, and $p \in [1,1+\frac{8}{\kappa_+}[$. Then (with the above notation) if $\lambda_\kappa \ge 1$, then
\begin{align*}
   \ex |\gamma(t,\kappa)-\gamma(s,\kappa)|^{\lambda_\kappa} &\le C (a(t,\zeta_\kappa)+a(s,\zeta_\kappa)) \, |t-s|^{(\zeta_\kappa+\lambda_\kappa)/2},\\
   \ex |\gamma(s,\kappa)-\gamma(s,\tilde\kappa)|^p &\le C|\sqrt{\kappa}-\sqrt{\tilde\kappa}|^p,
\end{align*}
where $C<\infty$ depends on $\kappa_-$, $\kappa_+$, $p$, and the choice of $r_\kappa$ (see above).
\end{proposition}

\begin{remark}
Note that $|\sqrt{\kappa}-\sqrt{\tilde\kappa}| \le C|\kappa-\tilde\kappa|$ if $\kappa,\tilde\kappa$ are bounded away from $0$.
\end{remark}

The first estimate is just \cite[Lemma 3.2]{FT17}.

The second estimate follows from the following result (which we will prove in \cref{sec:f_diffkappa_moment_pf}) and Fatou's lemma.

\begin{proposition}\label{thm:f_diffkappa_moment}
Let $0 < \kappa_- < \kappa_+ < \infty$ and $\kappa,\tilde\kappa \in [\kappa_-,\kappa_+]$. Let $t \in [0,T]$, $\delta \in {]0,1]}$, and $|x| \le \delta$. Then, for $1 \le p < 1+\frac{8}{\kappa_+}$, there exists $C < \infty$, depending on $\kappa_-$, $\kappa_+$, $T$, and $p$, such that
\begin{equation*}
    \ex |\hat f^\kappa_t(x+i\delta)-\hat f^{\tilde\kappa}_t(x+i\delta)|^p \le C|\sqrt{\kappa}-\sqrt{\tilde\kappa}|^p.
\end{equation*}
If $p > 1+\frac{8}{\kappa_+}$, then for any $\varepsilon > 0$ there exists $C < \infty$, depending on $\kappa_-$, $\kappa_+$, $T$, $p$, and $\varepsilon$, such that
\begin{equation*}
    \ex |\hat f^\kappa_t(x+i\delta)-\hat f^{\tilde\kappa}_t(x+i\delta)|^p \le C|\sqrt{\kappa}-\sqrt{\tilde\kappa}|^p \delta^{1+\frac{8}{\kappa_+}-p-\varepsilon}.
\end{equation*}
\end{proposition}

\medskip
\begin{remark}\label{rm:grr_application_jvrw}
Following the proof of \cite{JVRW14}, in particular using \cite[Lemma 2.3]{JVRW14} and \cref{thm:moment_estimate}, we can show
\begin{equation*}
    \ex |\hat f^\kappa_t(x+i\delta)-\hat f^{\tilde\kappa}_t(x+i\delta)|^{2\lambda-\varepsilon} \le C |\sqrt{\kappa}-\sqrt{\tilde\kappa}|^{2\lambda-\varepsilon} \delta^{-\lambda+\zeta-\varepsilon}.
\end{equation*}
If we use this estimate instead, we can estimate
\begin{align*}
    |\gamma(t,\kappa)-\gamma(s,\tilde\kappa)| &\le |\gamma(t,\kappa)-\gamma(s,\kappa)| + |\gamma(s,\kappa)-\gamma(s,\tilde\kappa)|\\
    &\le |\gamma(t,\kappa)-\gamma(s,\kappa)| \\
    &\quad + |\gamma(s,\kappa)-\hat f^\kappa_s(iy)| + |\hat f^\kappa_s(iy)-\hat f^{\tilde\kappa}_s(iy)| + |\hat f^{\tilde\kappa}_s(iy)-\gamma(s,\tilde\kappa)|
\end{align*}
with $y = |\Delta \kappa|$. Then, with
\begin{align*}
   \ex |\gamma(t,\kappa)-\gamma(s,\kappa)|^\lambda &\le C |t-s|^{(\zeta+\lambda)/2},\\
   \ex |\gamma(s,\kappa)-\hat f^\kappa_s(iy)|^\lambda &\le Cy^{\zeta+\lambda} = C |\kappa-\tilde\kappa|^{\zeta+\lambda},\\
   \ex |\hat f^\kappa_s(iy)-\hat f^{\tilde\kappa}_s(iy)|^{2\lambda-\varepsilon} &\le C |\kappa-\tilde\kappa|^{\zeta+\lambda-\varepsilon},
\end{align*}
\cref{thm:kolmogorov} applies if $(\frac{\zeta+\lambda}{2})^{-1}+(\zeta+\lambda)^{-1} < 1 \iff \zeta+\lambda > 3$, which happens when $\kappa \in {[0,8(2-\sqrt{3})[} \cup {]8(2+\sqrt{3}),\infty[}$ and with an appropriate choice of $r$. Hence, we recover the continuity of SLE in the same range as in \cite{JVRW14}.

Notice that for fixed $\kappa > 0$ the maximal value that $\zeta+\lambda$ can attain is $\frac{\kappa}{4}\left(\frac{1}{2}+\frac{4}{\kappa}\right)^2$ which is (for $\kappa < 8$) less than $p = 1+\frac{8}{\kappa}$ as in our \cref{thm:sle_diff_moments}. In other words, \cref{thm:sle_diff_moments} is really an improvement to \cite{JVRW14}.
\end{remark}

Below we write $x^+ = x\vee 0$ for $x \in \RR$.
\begin{corollary}\label{thm:fprime_diffkappa_moment}
    Under the same conditions as in \cref{thm:f_diffkappa_moment} we have
    \begin{equation*}
        \ex |(\hat f^\kappa_t)'(i\delta)-(\hat f^{\tilde\kappa}_t)'(i\delta)|^p \le C|\sqrt{\kappa}-\sqrt{\tilde\kappa}|^p \delta^{-p-(p-1-\frac{8}{\tilde\kappa}+\varepsilon)^+}
    \end{equation*}
    where $C<\infty$ depends on $\kappa_-$, $\kappa_+$, $T$, $p$, and $\varepsilon$.
\end{corollary}

\begin{proof}
For a holomorphic function $f: \HH \to \HH$, Cauchy Integral Formula tells us that
\begin{equation*}
    f'(i\delta) = \frac{1}{i2\pi} \int_\alpha \frac{f(w)}{(w-i\delta)^2} \, dw
\end{equation*}
where we let $\alpha$ be a circle of radius $\delta/2$ around $i\delta$. Consequently,
\begin{equation*}
|(\hat f^\kappa_t)'(i\delta)-(\hat f^{\tilde\kappa}_t)'(i\delta)| \le \frac{1}{2\pi} \int_\alpha \frac{|\hat f^\kappa_t(w)-\hat f^{\tilde\kappa}_t(w)|}{\delta^2/4} \, |dw|.
\end{equation*}

For all $w$ on the circle $\alpha$ we have $\Im w \in [\delta/2,3\delta/2]$ and $\Re w \in [-\delta/2,\delta/2]$. Therefore \cref{thm:f_diffkappa_moment} implies
\begin{equation*}
\ex |\hat f^\kappa_t(w)-\hat f^{\tilde\kappa}_t(w)|^p \le C|\Delta \sqrt{\kappa}|^p \delta^{-(p-1-\frac{8}{\tilde\kappa}+\varepsilon)^+}.
\end{equation*}

By Minkowski's inequality,
\begin{equation*}
\ex |(\hat f^\kappa_t)'(i\delta)-(\hat f^{\tilde\kappa}_t)'(i\delta)|^p \le \left( \frac{1}{2\pi} \int_\alpha \frac{(\ex |\hat f^\kappa_t(w)-\hat f^{\tilde\kappa}_t(w)|^p)^{1/p}}{\delta^2/4} \, |dw| \right)^p,
\end{equation*}
and the result follows since the length of $\alpha$ is $\pi \delta$.
\end{proof}

With \cref{thm:sle_diff_moments}, we can now apply \cref{thm:kolmogorov} to construct a Hölder continuous version of the map $\gamma = \gamma(t,\kappa)$, whose Hölder constants have some finite moments.

There is just one detail we still have to take into consideration. In order to apply \cref{thm:kolmogorov}, we have to use one common exponent $\lambda$ on the entire range of $\kappa$ where we want to apply the GRR lemma. Of course, we can choose new values for $\lambda$ again when we consider a different range of $\kappa$.

Alternatively, we could formulate our GRR version to allow exponents to vary with the parameters. But this will not be necessary since we can break our desired interval for $\kappa$ into subintervals.

\begin{proof}[Proof of \cref{thm:gamma_existence_hoelder}]
Consider the joint \slek{} process in some range $\kappa \in [\kappa_-,\kappa_+]$. We can assume that the interval $[\kappa_-,\kappa_+]$ is so small that $\lambda(\kappa)$ and $\zeta(\kappa)$ are almost constant. Otherwise, break $[\kappa_-,\kappa_+]$ into small subintervals and consider each of them separately.

We perform the proof in three parts. First we construct a continuous version $\tilde\gamma$ of $\gamma$ using \cref{thm:kolmogorov}. Then, using \cref{thm:grr}, we show that $\tilde\gamma$ is jointly H\"older continuous in both variables. Finally, we show that for each $\kappa$, the path $\tilde\gamma(\cdot,\kappa)$ is indeed the \slek{} trace generated by $\sqrt{\kappa}B$.

\textbf{Part 1.} For the first part, we would like to apply \cref{thm:kolmogorov}. There is just one technical detail we need to account for. In the estimates of \cref{thm:sle_diff_moments}, there is a singularity at time $t=0$, but we have not formulated \cref{thm:kolmogorov} to allow $C'$ to have a singularity. Therefore, it is easier to apply \cref{thm:kolmogorov} on the domain $[\varepsilon,1] \times [\kappa_-,\kappa_+]$ with $\varepsilon > 0$. With $\varepsilon \searrow 0$, we obtain a continuous version of $\gamma$ on the domain $]0,1] \times [\kappa_-,\kappa_+]$. Due to the local growth property of Loewner chains, we must have $\lim_{t \searrow 0} \gamma(t,\kappa) = 0$ uniformly in $\kappa$, so we actually have a continuous version of $\gamma$ on $[0,1] \times [\kappa_-,\kappa_+]$.
\footnote{Alternatively, we could also use the same strategy as in the proof of \cref{thm:kolmogorov}, and deduce the result directly from \cref{thm:grr}.}

Now we apply \cref{thm:sle_diff_moments} on the domain $[\varepsilon,1] \times [\kappa_-,\kappa_+]$. For this, we pick $\lambda \ge 1$, $r_\kappa < r_c(\kappa)$, and $p \in {[1,1+\frac{8}{\kappa_+}[}$ in such a way that $\lambda_\kappa = \lambda$ for all $\kappa \in [\kappa_-,\kappa_+]$. The condition to apply \cref{thm:kolmogorov} is then $(\frac{\zeta+\lambda}{2})^{-1}+p^{-1} < 1$.

A computation shows that $\zeta+\lambda = \frac{\kappa}{4}r\left(1+\frac{8}{\kappa}-r\right)$ attains its maximal value $\frac{\kappa}{4}\left(\frac{1}{2}+\frac{4}{\kappa}\right)^2$ at $r = \frac{1}{2}+\frac{4}{\kappa} = r_c$. Note also that $\lambda(r_c) = 1+\frac{2}{\kappa}+\frac{3}{32}\kappa > 1$. Recall from above that we can pick any $p < 1+\frac{8}{\kappa}$. Therefore, the condition for the exponents is
\begin{equation*}
    \frac{2}{\frac{\kappa}{4}\left(\frac{1}{2}+\frac{4}{\kappa}\right)^2}+\frac{1}{1+\frac{8}{\kappa}} < 1 \iff \kappa < \frac{8}{3}.
\end{equation*}
This completes the first part of the proof and gives us a continuous random field $\tilde\gamma$.

\textbf{Part 2.} Now that we have a random continuous function $\tilde\gamma$, we can apply \cref{thm:grr}. As in the proof of \cref{thm:kolmogorov}, we show that the integrals \eqref{eq:grr_condA} and \eqref{eq:grr_condB} have finite expectation, and therefore are almost surely finite. Denoting $\abs{A_1(t,s;\kappa)} \defeq \abs{\gamma(t,\kappa)-\gamma(s,\kappa)}$, $\abs{A_2(s;\kappa,\tilde\kappa)} \defeq \abs{\gamma(s,\kappa)-\gamma(s,\tilde\kappa)}$, and the corresponding integrals by $M_1, M_2$, we have by \cref{thm:sle_diff_moments}
\begin{align*}
\ex M_1 &\lesssim \iiint (a(t)+a(s)) \abs{t-s}^{(\zeta+\lambda)/2-\beta_1} \,dt\,ds\,d\kappa , \\
\ex M_2 &\lesssim \iiint \abs{\kappa-\tilde\kappa}^{p-\beta_2} \,ds\,d\kappa\,d\tilde\kappa .
\end{align*}
Picking $\beta_1 = \frac{\zeta+\lambda}{2}+1-\varepsilon$, $\beta_2 = p+1-\varepsilon$, the condition for the exponents is again $(\frac{\zeta+\lambda}{2})^{-1}+p^{-1} < 1$. Additionally, we need to account for the singularity at $t=0$ in the first integrand. This is not a problem if the function $a(t) = t^{-\zeta/2} \vee 1$ is integrable.

To make $a(t) = t^{-\zeta/2} \vee 1$ integrable, we would like to have $\zeta < 2$. \footnote{Alternatively, we can drop this condition if we make statements about the \slek{} process only on $t \in [\varepsilon, 1]$ for some $\varepsilon > 0$.} Recall that $\zeta = r-\frac{\kappa r^2}{8}$ from \eqref{eq:moment_estimate_parameters}. In case $\kappa > 1$, we always have $\zeta < 2$. In case $\kappa \le 1$, we have $\zeta < 2$ for $r < \frac{4}{\kappa}(1-\sqrt{1-\kappa})$, or equivalently $\lambda(r) < 3-\sqrt{1-\kappa}$. Therefore we can certainly find $r$ such that $\zeta < 2$ and $\zeta+\lambda \approx 2+(3-\sqrt{1-\kappa})$, and $p \approx 9 < 1+\frac{8}{\kappa}$. The condition $(\frac{\zeta+\lambda}{2})^{-1}+p^{-1} < 1$ is still fulfilled.

This proves the statements about the H\"older continuity of $\tilde\gamma$. 

\textbf{Part 3.} In the final part, we show that for each $\kappa$, the path $\tilde\gamma(\cdot,\kappa)$ is indeed the \slek{} trace generated by $\sqrt{\kappa}B$.

First, we fix a countable dense subset $\mathcal{K}$ in $[\kappa_-,\kappa_+]$. There exists a set $\Omega_1$ of probability 1 such that for all $\omega\in \Omega_1$, all $\kappa\in \mathcal{K}$, $\gamma(\kappa,t)$ exists and is continuous in $t$.

Since $\tilde\gamma$ is a version of $\gamma$, for all $t$,
\[ \pr\big( \gamma(t,\kappa)= \tilde\gamma(t,\kappa) \text{ for all } \kappa\in \mathcal{K}\big)=1.\] 
Hence, there exists a set $\Omega_2$ with probability 1 such that for all $\omega\in \Omega_2$, we have $\gamma(t,\kappa)=\tilde\gamma(t,\kappa)$ for all $\kappa\in \mathcal{K}$ and almost all $t$. 
Restricted to $\omega\in\Omega_3=\Omega_1\cap \Omega_2$, the previous statement is true for all $\kappa\in \mathcal{K}$ and all $t.$ We claim that on the set $\Omega_3$ of probability 1, the path $t \mapsto \tilde\gamma(t,\kappa)$ is indeed the \slek{} trace driven by $\sqrt{\kappa}B$. This can be shown in the same way as \cite[Theorem 4.7]{LSW04}.

Indeed, fix $t \in [0,1]$ and let $H_t = f^\kappa_t(\HH)$. We show that $H_t$ is the unbounded connected component of $\HH \setminus \tilde\gamma([0,t],\kappa)$ \footnote{Actually, there is only one component because it will turn out that $\tilde\gamma(\cdot,\kappa)$ is a simple trace.}. Find a sequence of $\kappa_n \in \mathcal{K}$ with $\kappa_n \to \kappa$ and let $(f^{\kappa_n}_t)$ be the corresponding inverse Loewner maps. Since $\sqrt{\kappa_n}B \to \sqrt{\kappa}B$, the Loewner differential equation implies that $f^{\kappa_n}_t \to f^\kappa_t$ uniformly on each compact set of $\HH$. By the chordal version of the Carathéodory kernel theorem (see \cite[Theorem 1.8]{Pom92}) which can be easily shown with the obvious adaptions, it follows that $H^{\kappa_n}_t \to H_t$ in the sense of kernel convergence. Since $\kappa_n\in\mathcal{K}$, we have $H^{\kappa_n}_t = \HH \setminus \gamma([0,t],\kappa_n) = \HH \setminus \tilde\gamma([0,t],\kappa_n)$. Therefore, the definitions of kernel convergence and the uniform continuity of $\tilde\gamma$ imply that $H_t$ is the unbounded connected component of $\HH \setminus \tilde\gamma([0,t],\kappa)$.
\end{proof}

By \cref{thm:gamma_existence_hoelder}, we now know that with probability one, the \slek{} trace $\gamma = \gamma(t,\kappa)$ is jointly continuous in $[0,1] \times [\kappa_-,\kappa_+]$. Similarly, applying \cref{thm:grr_pvar}, we can show the following.
\begin{theorem}\label{thm:sle_pvar_kappa}
Let $0 < \kappa_- < \kappa_+ < 8/3$. Let $\gamma^\kappa$ be the \slek{} trace driven by $\sqrt{\kappa}B$, and assume it is jointly continuous in $(t,\kappa) \in [0,1] \times [\kappa_-,\kappa_+]$. Consider $\gamma^\kappa$ as an element of $C^0([0,1])$ (with the metric $\|\cdot\|_\infty$). 

Then for some $0 < p <1/\eta$ (with $\eta$ from \cref{thm:gamma_existence_hoelder}), the $p$-variation of $\kappa \mapsto \gamma^\kappa$, $\kappa \in [\kappa_-,\kappa_+]$, is a.s. finite and bounded by some random variable $C$, depending on $\kappa_-$, $\kappa_+$, that has finite $\lambda$th moment for some $\lambda>1$.
\end{theorem}

We know that for fixed $\kappa \le 4$, the \slek{} trace is almost surely simple. It is natural to expect that there is a common set of probability $1$ where all \slek{} traces, $\kappa < 8/3$, are simple. This is indeed true.

\begin{theorem}
    Let $B$ be a standard Brownian motion. We have with probability $1$ that for all $\kappa < 8/3$ the \slek{} trace driven by $\sqrt{\kappa}B$ is simple.
\end{theorem}

\begin{proof}
    As shown in \cite[Theorem 6.1]{RS05}, due to the independent stationary increments of Brownian motion, this is equivalent to saying that $K^\kappa_t \cap \RR = \{0\}$ for all $t$ and $\kappa$, where $K^\kappa_t = \{ z \in \barH \mid T^\kappa_z \le t \}$ (the upper index denotes the dependence on $\kappa$). 
    
    Let $(g_t(x))_{t \ge 0}$ satisfy \eqref{eq:loewner} with $g_0(x)=x$ and driving function $U(t)=\sqrt{\kappa}B_t$. Then $X_t = \frac{g_t(x)-\sqrt{\kappa}B_t}{\sqrt{\kappa}}$ satisfies
    \begin{equation*}
        dX_t = \frac{2/\kappa}{X_t}\, dt - dB_t,
    \end{equation*}
    i.e. $X$ is a Bessel process of dimension $1+\frac{4}{\kappa}$. The statement $K^\kappa_t \cap \RR = \{0\}$ is equivalent to saying that $X_s \neq 0$ for all $x \neq 0$ and $s \in [0,t]$. This is a well-known property of Bessel processes, and stated in the lemma below.
\end{proof}

\begin{lemma}
    Let $B$ be a standard Brownian motion and suppose that we have a family of stochastic processes $X^{\kappa,x}$, $\kappa,x>0$, that satisfy
    \begin{equation*}
        X^{\kappa,x}_t = x + B_t + \int_0^t \frac{2/\kappa}{X^{\kappa,x}_s} \, ds, \quad t \in [0,T_{\kappa,x}]
    \end{equation*}
    where $T_{\kappa,x} = \inf\{ t \ge 0 \mid X^{\kappa,x}_t=0 \}$.
    
    Then we have with probability $1$ that $T_{\kappa,x}=\infty$ for all $\kappa \le 4$ and $x>0$.
\end{lemma}

\begin{proof}
    For fixed $\kappa \le 4$, see e.g. \cite[Proposition 1.21]{Law05}. To get the result simultaneously for all $\kappa$, use the property that if $\kappa < \tilde\kappa$ and $x>0$, then $X^{\kappa,x}_t > X^{\tilde\kappa,x}_t$ for all $t>0$, which follows from Gr\"onwall's inequality.
\end{proof}


\subsection{Stochastic continuity of \slek{} in $\kappa$}
\label{sec:stochastic_continuity}

In the previous section, we have shown almost sure continuity of \slek{} in $\kappa$ (in the range $\kappa \in [0,8/3[$). Weaker forms of continuity are easier to prove, and hold on a larger range of $\kappa$. We will show here that stochastic continuity (also continuity in $L^q(\pr)$ sense for some $q>1$ depending on $\kappa$) for all $\kappa \neq 8$ is an immediate consequence of our estimates. Below we write $\| f \|_{C^\alpha[a,b]} := \sup \frac{|f(t)-f(s)|}{|t-s|^\alpha}$, with $\sup$ taken over all $s<t$ in $[a,b]$.

\begin{theorem}\label{thm:sle_stochastic_continuity}
Let $\kappa > 0$, $\kappa \neq 8$. Then there exists $\alpha > 0$, $q > 1$, $r > 0$, and $C < \infty$ (depending on $\kappa$) such that if $\tilde\kappa$ is sufficiently close to $\kappa$ (where ``sufficiently close'' depends on $\kappa$), then
\begin{equation*}
    \ex \left[ \| \gamma(\cdot, \kappa)-\gamma(\cdot, \tilde\kappa) \|_{C^\alpha[0,1]}^q \right] \le C |\kappa - \tilde\kappa|^r .
\end{equation*}
In particular, if $\kappa_n \to \kappa$ exponentially fast, then $\| \gamma(\cdot, \kappa)-\gamma(\cdot, \kappa_n) \|_{C^\alpha[0,1]} \to 0$ almost surely.
\end{theorem}

Note that without sufficiently fast convergence of $\kappa_n \to \kappa$ it is not clear whether we can pass from $L^q$-convergence to almost sure convergence.

\begin{proof}
Fix $\kappa, \tilde\kappa \neq 8$. We apply \cref{thm:grr_1d} to the function $G: [0,1] \to \CC$, $G(t) = \gamma(t,\kappa)-\gamma(t,\tilde\kappa)$. We have
\begin{align*}
    |G(t)-G(s)| 
    &\le (|\gamma(t,\kappa)-\gamma(s,\kappa)|+|\gamma(t,\tilde\kappa)-\gamma(s,\tilde\kappa)|) \, 1_{|t-s| \le |\kappa-\tilde\kappa|} \\
    &\qquad + (|\gamma(t,\kappa)-\gamma(t,\tilde\kappa)|+|\gamma(s,\kappa)-\gamma(s,\tilde\kappa)|) \, 1_{|t-s| > |\kappa-\tilde\kappa|}\\
    &=: A_1(t,s) + A_2(t,s)
\end{align*}
where by \cref{thm:sle_diff_moments}
\begin{align*}
   \ex |A_1(t,s)|^\lambda &\le C (a^1(t)+a^1(s)) \, |t-s|^{(\zeta+\lambda)/2} \, 1_{|t-s| \le |\kappa-\tilde\kappa|},\\
   \ex |A_2(t,s)|^p &\le C|\kappa-\tilde\kappa|^p \, 1_{|t-s| > |\kappa-\tilde\kappa|},
\end{align*}
for suitable $\lambda \ge 1$, $p \in [1,1+\frac{8}{\kappa}[$.

It follows that, for $\beta_1,\beta_2 > 0$,
\begin{align*}
    \ex \iint \frac{|A_1(t,s)|^\lambda}{|t-s|^{\beta_1}} \, dt \, ds 
    &\le C \iint_{|t-s| \le |\kappa-\tilde\kappa|} (a^1(t)+a^1(s)) \, |t-s|^{(\zeta+\lambda)/2-\beta_1} \, dt \, ds \\
    &\le C |\kappa-\tilde\kappa|^{(\zeta+\lambda)/2-\beta_1+1},\\
    \ex \iint \frac{|A_2(t,s)|^p}{|t-s|^{\beta_2}} \, dt \, ds 
    &\le C|\kappa-\tilde\kappa|^p \iint_{|t-s| > |\kappa-\tilde\kappa|} |t-s|^{-\beta_2} \, dt \, ds \\
    &\le C|\kappa-\tilde\kappa|^{p-\beta_2+1}
\end{align*}
if $\zeta < 2$ and $\beta_1 < \frac{\zeta+\lambda}{2}+1$.

Recall that if $\kappa \neq 8$ and $\tilde\kappa$ is sufficiently close to $\kappa$, then the parameters $\lambda,\zeta$ are almost the same for $\kappa$ and $\tilde\kappa$, and (see the proof of \cref{thm:gamma_existence_hoelder}) they can be picked such that $\zeta < 2$ and $\zeta+\lambda > 2$. Hence, we can pick $\beta_1,\beta_2 > 2$ such that $2 < \beta_1 < \frac{\zeta+\lambda}{2}+1$ and $2 < \beta_2 < 1+p < 2+\frac{8}{\kappa}$.

The result follows from \cref{thm:grr_1d}, where we take $\alpha = \frac{\beta_1-2}{\lambda} \wedge \frac{\beta_2-2}{p}$ and $q = \lambda \wedge p$, which implies
\[
\ex \left[ \| G \|_{C^\alpha[0,1]}^q \right] \le C \ex\left[ \left( \iint \frac{|A_1(t,s)|^\lambda}{|t-s|^{\beta_1}} \, dt \, ds \right)^{q/\lambda} + \left( \iint \frac{|A_2(t,s)|^p}{|t-s|^{\beta_2}} \, dt \, ds \right)^{q/p} \right] .
\]
\end{proof}

\begin{corollary}\label{co:sle_stochastic_continuity}
For any $\kappa > 0$, $\kappa \neq 8$ and any sequence $\kappa_n \to \kappa$ we then have $\|\gamma^\kappa-\gamma^{\kappa_n}\|_{p\text{-var};[0,1]} \to 0$ in probability, for any $p > (1 + \kappa / 8) \wedge 2$. 
\end{corollary}

\begin{proof}
\cref{thm:sle_stochastic_continuity} immediately implies the statement with $\norm{\cdot}_\infty$. To upgrade the result to Hölder and $p$-variation topologies, recall the following general fact which follows from the interpolation inequalities for Hölder and $p$-variation constants (see e.g. \cite[Proposition 5.5]{FV10}):

Suppose $X_n$, $X$ are continuous stochastic processes such that for every $\varepsilon > 0$ there exists $M>0$ such that $\pr(\|X_n\|_{p\text{-var};[0,T]} > M) < \varepsilon$ for all $n$. If $X_n \to X$ in probability with respect to the $\|\cdot\|_\infty$ topology, then also with respect to the $p'$-variation topology for any $p' > p$. The analogous statement holds for Hölder topologies with $\alpha' < \alpha \leq 1$.

In order to apply this fact, we can use \cite[Theorem 5.2 and 6.1]{FT17} which bound the moments of $\|\gamma\|_{p\text{-var}}$ and $\|\gamma\|_{C^\alpha}$. The values for $p$ and $\alpha$ have also been computed there.
\end{proof}


\section{Convergence results} 
\label{sec:fhat_convergence}

Here we prove a stronger version of \cref{thm:gamma_existence_hoelder}, namely uniform convergence (even convergence in Hölder sense) of $\hat f^\kappa_t(iy)$ as $y \searrow 0$. For this result, we really use the full power of \cref{thm:grr} (actually \cref{le:grr_relaxed} as we will explain later). We point out that this is a stronger result than \cref{thm:main}, and that our previous proofs of \cref{thm:main,thm:main2} do not rely on this section.

The Hölder continuity in \cref{thm:gamma_existence_hoelder} induces an (inhomogeneous) Hölder space, with (inhomogeneous) Hölder constant that we denote by
\[ \norm{\gamma}_{C^{\alpha,\eta}} \defeq \sup_{(t,\kappa) \neq (s,\tilde\kappa)} \frac{\abs{\gamma(t,\kappa)-\gamma(s,\tilde\kappa)}}{\abs{t-s}^\alpha+\abs{\kappa-\tilde\kappa}^\eta} . \]

As before, we write
\begin{equation*}
    v(t,\kappa,y) = \int_0^y |(\hat f^\kappa_t)'(iu)| \, du.
\end{equation*}

\begin{theorem}\label{thm:fv_hoelder_convergence}
Let $\kappa_- > 0$, $\kappa_+ < 8/3$. Then $\|v(\cdot,\cdot,y)\|_{\infty; [0,1] \times [\kappa_-,\kappa_+]} \searrow 0$ almost surely as $y \searrow 0$. In particular, $\hat f^\kappa_t(iy)$ converges uniformly in $(t,\kappa) \in [0,1] \times [\kappa_-,\kappa_+]$ as $y \searrow 0$.

Moreover, both functions converge also almost surely in the same Hölder space $C^{\alpha,\eta}([0,1] \times [\kappa_-,\kappa_+])$ as in \cref{thm:gamma_existence_hoelder}.

Moreover, the (random) H\"older constants of $v(\cdot,\cdot,y)$ and $(t,\kappa) \mapsto \abs{\gamma(t,\kappa)-\hat f^\kappa_t(iy)}$ satisfy
\[ \ex[\norm{v(\cdot,\cdot,y)}_{C^{\alpha,\eta}}^\lambda] \le Cy^r \quad\text{and}\quad \ex[\norm{\gamma(\cdot,\cdot) - \hat f^\cdot_\cdot(iy)}_{C^{\alpha,\eta}}^\lambda] \le Cy^r \]
for some $\lambda > 1$, $r>0$ and $C<\infty$, and all $y \in {]0,1]}$.
\end{theorem}

As a consequence, we obtain also an improved version of \cite[Lemma 3.3]{JVRW14}.
\begin{corollary}
Let $\kappa_- > 0$, $\kappa_+ < 8/3$. Then there exist $\beta < 1$ and a random variable $c(\omega) < \infty$ such that almost surely
\[ \sup_{(t,\kappa) \in [0,1] \times [\kappa_-,\kappa_+]} \abs{(\hat f^\kappa_t)'(iy)} \le c(\omega) y^{-\beta} \]
for all $y \in {]0,1]}$.
\end{corollary}

\begin{proof}
By Koebe's $1/4$-Theorem we have $y\abs{(\hat f^\kappa_t)'(iy)} \le 4\dist(\hat f^\kappa_t(iy), \partial H^\kappa_t) \le 4v(t,\kappa,y)$. \cref{thm:fv_hoelder_convergence} and the Borel-Cantelli lemma imply
\[ \norm{v(\cdot,\cdot,2^{-n})}_\infty \le 2^{-nr'} \]
for some $r'>0$ and sufficiently large (depending on $\omega$) $n$. The result then follows by Koebe's distortion theorem (with $\beta = 1-r'$).
\end{proof}

The same method as \cref{thm:fv_hoelder_convergence} can be used to show the existence and H\"older continuity of the \slek{} trace for fixed $\kappa \neq 8$, avoiding a Borel-Cantelli argument. The best way of formulating this result is the terminology in \cite{FT17}.

For $\delta \in {]0,1[}$, $q \in {]1,\infty[}$, define the fractional Sobolev (Slobodeckij) semi-norm of a measurable function $x: [0,1] \to \CC$ as
\begin{align*}
\|x\|_{W^{\delta,q}} := \left( \int_0^1 \int_0^1 \frac{|x(t)-x(s)|^q}{|t-s|^{1+\delta q}} \, ds \, dt \right)^{1/q}.
\end{align*}
As a consequence of the (classical) one-dimensional GRR inequality (see \cite[Corollary A.2 and A.3]{FV10}), we have that for all $\delta \in {]0,1[}$, $q \in {]1,\infty[}$ with $\delta-1/q > 0$, there exists a constant $C<\infty$ such that for all $x \in C[0,1]$ we have
\begin{align*}
\|x\|_{C^\alpha[s,t]} \le C \|x\|_{W^{\delta,q}[s,t]}
\end{align*}
and
\begin{align*}
\|x\|_{p\text{-var};[s,t]} \le C |t-s|^\alpha \|x\|_{W^{\delta,q}[s,t]},
\end{align*}
where $p=1/\delta$ and $\alpha=\delta-1/q$, and $\|x\|_{C^\alpha[s,t]}$ and $\|x\|_{p\text{-var};[s,t]}$ denote the H\"older and $p$-variation constants of $x$, restricted to $[s,t]$.

Fix $\kappa \ge 0$, and as before, let
\begin{equation*}
    v(t,y) = \int_0^y |\hat f_t'(iu)| \, du.
\end{equation*}
Recall the notation \eqref{eq:moment_estimate_parameters}, and let $\lambda=\lambda(r)$, $\zeta=\zeta(r)$ with some $r < r_c(\kappa)$.

The following result is proved similarly to \cref{thm:fv_hoelder_convergence}.
\begin{theorem}
Let $\kappa \neq 8$. Then for some $\alpha > 0$ and some $p < 1/\alpha$ there almost surely exists a continuous $\gamma: [0,1] \to \barH$ such that the function $t \mapsto \hat f_t(iy)$ converges in $C^\alpha$ and $p$-variation to $\gamma$ as $y \searrow 0$.

More precisely, let $\kappa \ge 0$ be arbitrary, $\zeta<2$ and $\delta \in {\left]0, \frac{\lambda+\zeta}{2\lambda} \right[}$. Then there exists a random measurable function $\gamma: [0,1] \to \barH$ such that
\begin{equation*}
    \ex\|v(\cdot,y)\|_{W^{\delta,\lambda}}^{\lambda} \le C y^{\lambda+\zeta-2\delta\lambda} \quad \text{and} \quad \ex\|\gamma-\hat f_\cdot(iy)\|_{W^{\delta,\lambda}}^{\lambda} \le C y^{\lambda+\zeta-2\delta\lambda}
\end{equation*}
for all $y \in {]0,1]}$, where $C$ is a constant that depends on $\kappa$, $r$, and $\delta$. Moreover, a.s. $\|v(\cdot,y)\|_{W^{\delta,\lambda}} \to 0$ and $\|\gamma-\hat f_\cdot(iy)\|_{W^{\delta,\lambda}} \to 0$ as $y \searrow 0$.

If additionally $\delta \in {\left]\frac{1}{\lambda}, \frac{\lambda+\zeta}{2\lambda} \right[}$, then the same is true for $\|\cdot\|_{1/\delta\text{-var}}$ and $\|\cdot\|_{C^\alpha}$ where $\alpha=\delta-1/\lambda$.
\end{theorem}

\begin{remark}
The conditions for the exponents are the same as in \cite{FT17}. In particular, the result applies to the (for \slek{}) optimal $p$-variation and H\"older exponents.
\end{remark}

\begin{proof}[Proof of \cref{thm:fv_hoelder_convergence}]
We use the same setting as in the proof of \cref{thm:gamma_existence_hoelder}. For $\kappa \le \kappa_+ < 8/3$, we choose $p \in [1, 1+\frac{8}{\kappa_+}[$, $r_\kappa < r_c(\kappa)$, $\lambda(\kappa,r_\kappa) = \lambda \ge 1$, and the corresponding $\zeta_\kappa = \zeta(\kappa,r_\kappa)$ as in the proof of \cref{thm:gamma_existence_hoelder}. Again, we assume that the interval $[\kappa_-,\kappa_+]$ is small enough so that $\lambda(\kappa)$ and $\zeta(\kappa)$ are almost constant.

\textbf{Step 1.} We would like to show that $v$ and $\hat f$ (defined above) are Cauchy sequences in the aforementioned Hölder space as $y \searrow 0$. Therefore we will take differences $|v(\cdot,\cdot,y_1)-v(\cdot,\cdot,y_2)|$ and $|\hat f(iy_1)-\hat f(iy_2)|$, and estimate their Hölder norms with our GRR lemma. Note that it is not a priori clear that $v(t,\kappa,y)$ is continuous in $(t,\kappa)$, but $|v(t,\kappa,y_1)-v(t,\kappa,y_2)| = \int_{y_1}^{y_2} |(\hat f^\kappa_t)'(iu)| \, du$ certainly is, so the GRR lemma can be applied to this function.

Consider the function
\begin{equation*}
    G(t,\kappa) := v(t,\kappa,y)-v(t,\kappa,y_1) = \int_{y_1}^y |(\hat f^\kappa_t)'(iu)| \, du.
\end{equation*}

The strategy will be to show that the condition of \cref{thm:grr} is satisfied almost surely for $G$. As in the proof of Kolmogorov's continuity theorem, we do this by showing that the expectation of the integrals \eqref{eq:grr_condA}, \eqref{eq:grr_condB} are finite (after defining suitable $A_{1j}$, $A_{2j}$) and converge to $0$ as $y \searrow 0$. In particular, they are almost surely finite, so \cref{thm:grr} then implies that $G$ is Hölder continuous, with Hölder constant bounded in terms of the integrals \eqref{eq:grr_condA}, \eqref{eq:grr_condB}.

We would like to infer that almost surely the functions $v(\cdot,\cdot,y)$, $y>0$, form a Cauchy sequence in the Hölder space $C^{\alpha,\eta}$. But this is not immediately clear, therefore we will bound the integrals \eqref{eq:grr_condA}, \eqref{eq:grr_condB} by expressions that are decreasing in $y$. We will also define $A_{1j}$, $A_{2j}$ here.

In order to do so, we estimate
\begin{align*}
&|G(t,\kappa)-G(s,\tilde\kappa)| \\
&\quad \le \int_0^y \left| |(\hat f^\kappa_t)'(iu)|-|(\hat f^\kappa_s)'(iu)| \right| \, du + \int_0^y \left| |(\hat f^\kappa_s)'(iu)|-|(\hat f^{\tilde\kappa}_s)'(iu)| \right| \, du\\
&\quad \le \int_0^y |(\hat f^\kappa_t)'(iu)-(\hat f^\kappa_s)'(iu)| \, du + \int_0^y |(\hat f^\kappa_s)'(iu)-(\hat f^{\tilde\kappa}_s)'(iu)| \, du\\
&\quad =: A_{1*}(t,s;\kappa)+A_{2*}(s;\kappa,\tilde\kappa),
\end{align*}

Moreover, the function $\hat G(t,\kappa) := \hat f^\kappa_t(iy)-\hat f^\kappa_t(iy_1)$ also satisfies
\begin{equation*}
|\hat G(t,\kappa)-\hat G(s,\tilde\kappa)| \le A_{1*}(t,s;\kappa)+A_{2*}(s;\kappa,\tilde\kappa).
\end{equation*}
Therefore all our considerations for $G$ apply also to $\hat G$.

We want to estimate the difference $|(\hat f^\kappa_s)'(iu)-(\hat f^{\tilde\kappa}_s)'(iu)|$ differently for small and large $u$ (relatively to $|\Delta\kappa|$), therefore we we split $A_{2*}$ into
\begin{align*}
A_{2*}(s;\kappa,\tilde\kappa) 
&= \int_0^{y \wedge |\kappa-\tilde\kappa|^{p/(\zeta+\lambda)}} |(\hat f^\kappa_s)'(iu)-(\hat f^{\tilde\kappa}_s)'(iu)| \, du \\
&\qquad + \int_{y \wedge |\kappa-\tilde\kappa|^{p/(\zeta+\lambda)}}^y |(\hat f^\kappa_s)'(iu)-(\hat f^{\tilde\kappa}_s)'(iu)| \, du\\
&=: A_{21}(s;\kappa,\tilde\kappa)\\
&\qquad +A_{22}(s;\kappa,\tilde\kappa).
\end{align*}

We would like to apply \cref{thm:grr} with these choices of $A_{1*}, A_{21}, A_{22}$. We denote the integrals \eqref{eq:grr_condA}, \eqref{eq:grr_condB} by
\begin{align*}
    M_{1*} &\defeq \iiint \frac{|A_{1*}(t,s;\kappa)|^\lambda}{|t-s|^{\beta_1}} \, ds \, dt \, d\kappa,\\
    M_{21} &\defeq \iiint \frac{|A_{21}(s;\kappa,\tilde\kappa)|^\lambda}{|\kappa-\tilde\kappa|^{\beta_2}} \, ds \, d\kappa \, d\tilde\kappa,\\
    M_{22} &\defeq \iiint \frac{|A_{22}(s;\kappa,\tilde\kappa)|^p}{|\kappa-\tilde\kappa|^{\beta_2}} \, ds \, d\kappa \, d\tilde\kappa.
\end{align*}
Suppose that we can show that
\[ \ex[M_{1*}] \lesssim y^r, \quad \ex[M_{2j}] \lesssim y^r  \]
for some $r>0$. This would imply that they are almost surely finite, and that $G$ and $\hat G$ are Hölder continuous with $\norm{G}_{C^{\alpha,\eta}} \lesssim M_{A*}^{1/\lambda}+M_{21}^{1/\lambda}+M_{22}^{1/p}$ (same for $\hat G$).

Notice that now $A_{1*}, A_{21}, A_{22}$, hence also $M_{A*}, M_{21}, M_{22}$ are decreasing in $y$. So as we let $y,y_1 \searrow 0$, it would follow that
\begin{itemize}
    \item $\ex[\norm{G}_{C^{\alpha,\eta}}^\lambda] \lesssim y^{r'} \to 0$ (same for $\hat G$) with a (possibly) different $r'>0$. In particular, as $y \searrow 0$, the random functions $v(\cdot,\cdot,y)$ and $(t,\kappa) \mapsto \hat f^\kappa_t(iy)$ form Cauchy sequences in $L^\lambda(\pr;C^{\alpha,\eta})$, and it follows that also $\ex[\norm{v(\cdot,\cdot,y)}_{C^{\alpha,\eta}}^\lambda] \lesssim y^{r'} \to 0$ and $\ex[\norm{\gamma(\cdot,\cdot) - \hat f^\cdot_\cdot(iy)}_{C^{\alpha,\eta}}^\lambda] \lesssim y^{r'} \to 0$ as $y \searrow 0$.
    \item By the monotonicity of $M_{A*}, M_{21}, M_{22}$ in $y$ we have that almost surely the functions $v(\cdot,\cdot,y)$ and $(t,\kappa) \mapsto \hat f^\kappa_t(iy)$ are Cauchy sequences in the Hölder space $C^{\alpha,\eta}$.
\end{itemize}
This will show \cref{thm:fv_hoelder_convergence}.

\textbf{Step 2.} We now explain that in fact, our definition of $A_{1*}$ does not always suffice, and we need to define $A_{1j}$ a bit differently in order to get the best estimates. The new definition of $A_{1j}$ will satisfy only the relaxed condition \eqref{e:G_assump_weaker} (instead of \eqref{e:G_assump}).

The reason is that, when $|t-s| \le u^2$, $|\hat f_t(iu)-\hat f_s(iu)|$ is estimated by an expression like $|\hat f_s'(iu)| |B_t-B_s|$ which is of the order $O(|t-s|^{1/2})$. The same is true for the difference $|\hat f_t'(iu)-\hat f_s'(iu)|$ (see \eqref{eq:hatfprime_difftime} below). When we carry out the moment estimate for our choice of $A_{1*}$, then we will get 
\[ \ex|A_{1*}(t,s;\kappa)|^\lambda = O(|t-s|^{\lambda/2}) .\]
But recall from \cref{thm:sle_diff_moments} that 
\[ \ex|\gamma(t)-\gamma(s)|^\lambda \le C|t-s|^{(\zeta+\lambda)/2} ,\]
which has allowed us to apply \cref{thm:grr} with $\beta_1 \approx \frac{\zeta+\lambda}{2}+1$ in the proof of \cref{thm:gamma_existence_hoelder}. When $\zeta>0$, this was better than just $\lambda/2$.

To fix this, we need to adjust our choice of $A_{1j}$. In particular, we should not evaluate $\ex|\hat f_t'(iu)-\hat f_s'(iu)|^\lambda$ when $u \gg |t-s|^{1/2}$ (here ``$\gg$'' means ``much larger''). As observed in \cite{JVL11}, $|\hat f_s'(iu)|$ does not change much in time when $u \gg |t-s|^{1/2}$. More precisely, we have the following results.
\begin{lemma}
    Let $(g_t)$ be a chordal Loewner chain driven by $U$, and $\hat f_t(z) = g_t^{-1}(z+U(t))$. Then, if $t,s \ge 0$ and $z=x+iy \in \HH$ such that $|t-s| \le C' y^2$, we have
    \begin{align}
        |\hat f_t'(z)| &\le C|\hat f_s'(z)| \left( 1+\frac{|U(t)-U(s)|^2}{y^2} \right)^l, \label{eq:hatfprime_changetime}\\
        |\hat f_t(z)-\hat f_s(z)| &\le C |\hat f_s'(z)| \left( \frac{|t-s|}{y} + |U(t)-U(s)| \left( 1+\frac{|U(t)-U(s)|^2}{y^2} \right)^l \right), \label{eq:hatf_difftime}\\
        |\hat f_t'(z)-\hat f_s'(z)| &\le C |\hat f_s'(z)| \left( \frac{|t-s|}{y^2} + \frac{|U(t)-U(s)|}{y} \left( 1+\frac{|U(t)-U(s)|^2}{y^2} \right)^l \right), \label{eq:hatfprime_difftime}
    \end{align}
    where $C < \infty$ depends on $C' < \infty$, and $l < \infty$ is a universal constant.
\end{lemma}

\begin{proof}
The first two inequalities \eqref{eq:hatfprime_changetime} and \eqref{eq:hatf_difftime} follow from \cite[Lemma 3.5 and 3.2]{JVL11}. The third inequality \eqref{eq:hatfprime_difftime} follows from \eqref{eq:hatf_difftime} by the Cauchy integral formula in the same way as in \cref{thm:fprime_diffkappa_moment}. Note that for $z \in \HH$ and $w$ on a circle of radius $y/2$ around $z$, we have $|\hat f_s'(w)| \le 12 |\hat f_s'(z)|$ by the Koebe distortion theorem.
\end{proof}

We now redefine $A_{1j}$. Let
\begin{align*}
A_{11}(t,s;\kappa) &= \int_0^{y \wedge |t-s|^{1/2}} |\hat f_t'(iu)-\hat f_s'(iu)| \, du,\\
A_{12}(t,s;\kappa) &= \int_{y \wedge |t-s|^{1/2}}^y \frac{|t-s|}{u^2} |\hat f_s'(iu)| \, du,\\
A_{13}(t,s;\kappa) &= \int_{y \wedge |t-s|^{1/2}}^{y \wedge 2|t-s|^{1/2}} u^{-1} |\hat f_s'(iu)| \left(1+\norm{B}_{C^{1/2^{(-)}}}\right)^{2l+1} |t-s|^{1/2^{(-)}} \, du,
\end{align*}
for $s \le t$, where the exponents $1/2^{(-)} < 1/2$ denote some numbers that we can pick arbitrarily close to $1/2$. (Of course, $\hat f_t$ still depends on $\kappa$, but for convenience we do not write it for now.)

Note that the integrands in $A_{12}$ and $A_{13}$ just make fancy bounds of 
\[ |\hat f_t'(iu)-\hat f_s'(iu)| ,\]
according to \eqref{eq:hatfprime_difftime}. But now, in $A_{13}$ we are not integrating up to $y$ any more. Thus, the condition \eqref{e:G_assump} is not satisfied any more. But the relaxed condition \eqref{e:G_assump_weaker} of \cref{le:grr_relaxed} is still satisfied. Indeed, by \eqref{eq:hatfprime_difftime},
\[ \begin{split}
A_{1*}(t,s;\kappa) &\le A_{11}(t,s;\kappa) + \int_{y \wedge \abs{t-s}^{1/2}}^y |\hat f_t'(iu)-\hat f_s'(iu)| \, du \\
&\le A_{11}(t,s;\kappa) + A_{12}(t,s;\kappa) \\
&\quad + \int_{y \wedge \abs{t-s}^{1/2}}^y u^{-1} |\hat f_s'(iu)| \left(1+\norm{B}_{C^{1/2^{(-)}}}\right)^{l+1} |t-s|^{1/2^{(-)}} \, du
\end{split} \]
where by \eqref{eq:hatfprime_changetime}
\[ \begin{split}
&\int_{y \wedge \abs{t-s}^{1/2}}^y u^{-1} |\hat f_s'(iu)| \left(1+\norm{B}_{C^{1/2^{(-)}}}\right)^{l+1} |t-s|^{1/2^{(-)}} \, du \\
&\quad = \sum_{k=0}^{\lfloor \log_4(y^2/\abs{t-s}) \rfloor} \int_{y \wedge (4^k\abs{t-s})^{1/2}}^{y \wedge 2(4^k\abs{t-s})^{1/2}} ... \\
&\quad = \sum_{k=0}^{\lfloor \log_4(y^2/\abs{t-s}) \rfloor} 4^{-k(1/2^{(-)})} \abs{A_{13}(t_1+4^k(t-t_1),t_1+4^k(s-t_1);\kappa)}
\end{split} \]
whenever $\abs{s-t_1} \le 2\abs{t-s}$ (implying $\abs{s-(t_1+4^k(s-t_1))} \le (4^k-1)2\abs{t-s} \le 2u^2$).

Finally, with this definition of $A_{13}$, we truly have $\ex |A_{13}(t,s;\kappa)|^{\lambda^{(-)}} = O(|t-s|^{(\zeta+\lambda)^{(-)}/2})$ and not just $O(|t-s|^{\lambda/2})$; here $\lambda^{(-)} < \lambda$ is an exponent that can be chosen arbitrarily close to $\lambda$.
\begin{proposition}\label{thm:grr_integrals_expectation}
With the above notation and assumptions, if $1 < \beta_1 < \frac{\zeta+\lambda}{2}+1$, $1 < \beta_2 < p+1$, we have
\begin{align*}
\ex \iiint \frac{|A_{1j}(t,s;\kappa)|^\lambda}{|t-s|^{\beta_1}} \, ds \, dt \, d\kappa &\le C y^{\zeta+\lambda-2\beta_1+2} \iint a(s,\zeta_\kappa) \, ds \, d\kappa, \quad j=1,2,\\
\ex \iiint \frac{|A_{13}(t,s;\kappa)|^{\lambda^{(-)}}}{|t-s|^{\beta_1}} \, ds \, dt \, d\kappa &\le C y^{(\zeta+\lambda)^{(-)} -2\beta_1+2} \iint a(s,\zeta_\kappa)^{1^{(-)}} \, ds \, d\kappa,\\
\ex \iiint \frac{|A_{21}(s;\kappa,\tilde\kappa)|^\lambda}{|\kappa-\tilde\kappa|^{\beta_2}} \, ds \, d\kappa \, d\tilde\kappa &\le C y^{(\zeta+\lambda)(p-\beta_2+1)/p} \iint a(s,\zeta_\kappa) \, ds \, d\kappa,\\
\ex \iiint \frac{|A_{22}(s;\kappa,\tilde\kappa)|^p}{|\kappa-\tilde\kappa|^{\beta_2}} \, ds \, d\kappa \, d\tilde\kappa &\le C y^{(\zeta+\lambda)(p-\beta_2+1)/p},
\end{align*}
where $C$ depends on $\kappa_-$, $\kappa_+$, $\lambda$, $p$, $\beta_1$, $\beta_2$.
\end{proposition}

\begin{proof}
These follow from direct computations making use of \cref{thm:moment_estimate} and \cref{thm:fprime_diffkappa_moment}. They can be found in the appendix of the arXiv version of this paper.
\end{proof}

Recall that the condition for \cref{thm:grr} is $(\beta_1-2)(\beta_2-2)-1 > 0$. With $\beta_1 < \frac{\lambda+\zeta}{2}+1$, $\beta_2 < p+1$ this is again the condition $(\frac{\zeta+\lambda}{2})^{-1}+p^{-1} < 1$, which leads to $\kappa < \frac{8}{3}$. Moreover, we need the additional condition $\frac{\beta_1-2}{\lambda} < 1/2^{(-)}$ for \cref{le:grr_relaxed}, which is implied by $\zeta < 2$.

The same analysis of $\lambda$ and $\zeta$ as in the proof of \cref{thm:gamma_existence_hoelder} applies here. This finishes the proof of \cref{thm:fv_hoelder_convergence}. 
\end{proof}


\section{Proof of \cref{thm:f_diffkappa_moment}}
\label{sec:f_diffkappa_moment_pf}

The proof is based on the methods of \cite{Law09,JVRW14}.

Let $t \ge 0$ and $U \in C([0,t];\RR)$. We study the chordal Loewner chain $(g_s)_{s \in [0,t]}$ in $\HH$ driven by $U$, i.e. the solution of \eqref{eq:loewner}. Let $V(s) = U(t-s)-U(t)$, $s \in [0,t]$, and consider the solution of the reverse flow
\begin{equation}\label{eq:loewner_reverse_2}
    \partial_s h_s(z) = \frac{-2}{h_s(z)-V(s)}, \quad h_0(z) = z.
\end{equation}
The Loewner equation implies $h_t(z) = g_t^{-1}(z+U(t))-U(t) = \hat f_t(z)-U(t)$.

Let $x_s + iy_s = z_s = z_s(z) = h_s(z)-V(s)$. Recall that
\begin{equation*}
    \partial_s \log |h_s'(z)| = 2 \frac{x_s^2-y_s^2}{(x_s^2+y_s^2)^2}
\end{equation*}
and therefore
\begin{equation*}
    |h_s'(z)| = \exp\left( 2 \int_0^s \frac{x_\vartheta^2-y_\vartheta^2}{(x_\vartheta^2+y_\vartheta^2)^2} \, d\vartheta \right).
\end{equation*}

For $r \in [0,t]$, denote by $h_{r,s}$ the reverse Loewner flow driven by $V(s)-V(r)$, $s \in [r,t]$. More specifically,
\begin{align*}
    \partial_s (h_{r,s}(z_r(z))+V(r)) &= \frac{-2}{(h_{r,s}(z_r(z))+V(r))-V(s)},\\
    h_{r,r}(z_r(z))+V(r) &= z_r(z)+V(r) = h_r(z),
\end{align*}
which implies from (\ref{eq:loewner_reverse_2}) that
\begin{alignat*}{2}
    && h_{r,s}(z_r(z))+V(r) &= h_s(z)\\
    &\text{and}& z_{r,s}(z_r(z)) &= z_s(z) \quad \text{for all } s \in [r,t].
\end{alignat*}

\noindent This implies also
\begin{equation*}
    |h_{r,s}'(z_r(z))| = \exp\left( 2 \int_r^s \frac{x_\vartheta^2-y_\vartheta^2}{(x_\vartheta^2+y_\vartheta^2)^2} \, d\vartheta \right).
\end{equation*}

The following result is essentially \cite[Lemma 2.3]{JVRW14}, stated in a more refined way.
\begin{lemma}\label{thm:h_diff}
    Let $V^1, V^2 \in C([0,t];\RR)$, and denote by $(h^j_s)$ the reverse Loewner flow driven by $V^j$, $j=1,2$, respectively. For $z=x+iy$, denoting $x^j_s + iy^j_s = z^j_s = h^j_s(z)-V^j(s)$, we have
    \begin{multline*}
        |h^1_t(z)-h^2_t(z)| \\
        \le 2(y^2+4t)^{1/4} \int_0^t |V^1(s)-V^2(s)| \frac{1}{|z^1_s z^2_s|} \frac{1}{(y^1_s y^2_s)^{1/4}} |(h^1_{s,t})'(z^1_s) (h^2_{s,t})'(z^2_s)|^{1/4} \, ds.
    \end{multline*}
\end{lemma}

\begin{proof}
    The proof of \cite[Lemma 2.3]{JVRW14} shows that
    \begin{multline*}
        |h^1_t(z)-h^2_t(z)| \\
        \le \int_0^t |V^1(s)-V^2(s)| \frac{2}{|z^1_s z^2_s|} \exp\left( 2 \int_s^t \frac{x^1_\vartheta x^2_\vartheta - y^1_\vartheta y^2_\vartheta}{((x^1_\vartheta)^2+(y^1_\vartheta)^2)((x^2_\vartheta)^2+(y^2_\vartheta)^2)} \, d\vartheta \right) \, ds.
    \end{multline*}
    
    The claim follows by estimating
    \begin{align*}
        & 2 \int_s^t \frac{x^1_\vartheta x^2_\vartheta - y^1_\vartheta y^2_\vartheta}{((x^1_\vartheta)^2+(y^1_\vartheta)^2)((x^2_\vartheta)^2+(y^2_\vartheta)^2)} \, d\vartheta\\
        &\quad \le 2 \int_s^t \frac{x^1_\vartheta x^2_\vartheta}{((x^1_\vartheta)^2+(y^1_\vartheta)^2)((x^2_\vartheta)^2+(y^2_\vartheta)^2)} \, d\vartheta\\
        &\quad \le \prod_{j=1,2} \left( 2 \int_s^t \frac{(x^j_\vartheta)^2}{((x^j_\vartheta)^2+(y^j_\vartheta)^2)^2} \, d\vartheta \right)^{1/2}\\
        &\quad = \prod_{j=1,2} \left( \frac{1}{2} \int_s^t \frac{2((x^j_\vartheta)^2-(y^j_\vartheta)^2)}{((x^j_\vartheta)^2+(y^j_\vartheta)^2)^2} \, d\vartheta + \frac{1}{2} \int_s^t \frac{2}{(x^j_\vartheta)^2+(y^j_\vartheta)^2} \, d\vartheta \right)^{1/2}\\
        &\quad = \prod_{j=1,2} \left( \frac{1}{2} \log |(h^j_{s,t})'(z^j_s)| + \frac{1}{2} \log\frac{y^j_t}{y^j_s} \right)^{1/2}\\
        &\quad \le \sum_{j=1,2} \left( \frac{1}{4} \log |(h^j_{s,t})'(z^j_s)| + \frac{1}{4} \log\frac{y^j_t}{y^j_s} \right)
    \end{align*}
    and $y^j_t \le \sqrt{y^2+4t}$. (In the last line we used $\sqrt{ab}\le \frac{a+b}{2}$ for $a,b\ge 0$.)
\end{proof}

\subsection{Taking moments}

Let $\kappa,\tilde\kappa > 0$, and let $V^1 = \sqrt{\kappa}B$, $V^2 = \sqrt{\tilde\kappa}B$, where $B$ is a standard Brownian motion. In the following, $C$ will always denote a finite deterministic constant that might change from line to line.

\cref{thm:h_diff} and the Cauchy-Schwarz inequality imply
\begin{align}
    &\ex |h^1_t(z)-h^2_t(z)|^p \nonumber\\
    &\quad \le C |\Delta \sqrt{\kappa}|^p \, \ex \left| \int_0^t |B_s| \frac{1}{|z^1_s z^2_s|} \frac{1}{(y^1_s y^2_s)^{1/4}} |(h^1_{s,t})'(z^1_s) (h^2_{s,t})'(z^2_s)|^{1/4} \, ds \right|^p \nonumber\\
    &\quad \le C |\Delta \sqrt{\kappa}|^p \, \ex \prod_{j=1,2} \left| \int_0^t |B_s| \frac{1}{|z^j_s|^2} \frac{1}{(y^j_s)^{1/2}} |(h^j_{s,t})'(z^j_s)|^{1/2} \, ds \right|^{p/2} \nonumber\\
    &\quad \le C |\Delta \sqrt{\kappa}|^p \prod_{j=1,2} \left( \ex \left| \int_0^t |B_s| \frac{1}{|z^j_s|^2} \frac{1}{(y^j_s)^{1/2}} |(h^j_{s,t})'(z^j_s)|^{1/2} \, ds \right|^p \right)^{1/2}. \label{eq:h_diffkappa_moment}
\end{align}

Now the flows for $\kappa$ and $\tilde\kappa$ can be studied separately. We see that as long as the above integral is bounded, then $\ex |\Delta_{\sqrt{\kappa}} h^\kappa_t(z)|^p \lesssim |\Delta \sqrt{\kappa}|^p$. Heuristically, the typical growth of $y_s$ is like $\sqrt{s}$, as was shown in \cite{Law09}. Therefore, we expect the integrand to be bounded by $s^{1/2-1-1/4-\beta/4} = s^{-(3+\beta)/4}$ which is integrable since $\beta = \beta(\kappa) < 1$ for $\kappa \neq 8$.

In order to make the idea precise, we will reparametrise the integral in order to match the setting in \cite{Law09} and apply their results.

\subsection{Reparametrisation}

Let $\kappa > 0$. In \cite{Law09}, the flow
\begin{equation}\label{eq:loewner_reverse_a}
    \partial_s \tilde h_s(z) = \frac{-a}{\tilde h_s(z)-\tilde B_s}, \quad \tilde h_0(z) = z,
\end{equation}
with $a = \dfrac{2}{\kappa}$ is considered. To translate our notation, observe that
\begin{equation*}
    \partial_s h_{s/\kappa}(z) = \frac{-2/\kappa}{h_{s/\kappa}(z)-\sqrt{\kappa}B_{s/\kappa}}.
\end{equation*}
If we let $\tilde B_s = \sqrt{\kappa}B_{s/\kappa}$, then 
\begin{equation*}
    h_{s/\kappa}(z) = \tilde h_s(z) \implies h_s(z) = \tilde h_{\kappa s}(z).
\end{equation*}
Moreover, if we let $\tilde z_s = \tilde h_s(z) - \tilde B_s$, then $z_s = h_s(z) - \sqrt{\kappa}B_s = \tilde z_{\kappa s}$.

Therefore,
\begin{align*}
    \int_0^t |B_s| \frac{1}{|z_s|^2} \frac{1}{y_s^{1/2}} |h_{s,t}'(z_s)|^{1/2} \, ds 
    &= \int_0^t \left|\frac{1}{\sqrt{\kappa}}\tilde B_{\kappa s}\right| \frac{1}{|\tilde z_{\kappa s}|^2} \frac{1}{\tilde y_{\kappa s}^{1/2}} |\tilde h_{\kappa s,\kappa t}'(\tilde z_{\kappa s})|^{1/2} \, ds\\
    &= \int_0^{\kappa t} \kappa^{-3/2} |\tilde B_s| \frac{1}{|\tilde z_s|^2} \frac{1}{\tilde y_s^{1/2}} |\tilde h_{s,\kappa t}'(\tilde z_s)|^{1/2} \, ds.
\end{align*}

For notational simplicity, we will write just $t$ instead of $\kappa t$ and $B, h_s, z_s$ instead of $\tilde B, \tilde h_s, \tilde z_s$.

In the next step, we will let the flow start at $z_0 = i$ instead of $i\delta$. Observe that
\begin{equation*}
    \partial_s (\delta^{-1} h_{\delta^2 s}(\delta z)) = \frac{-a}{\delta^{-1} h_{\delta^2 s}(\delta z) - \delta^{-1} B_{\delta^2 s}},
\end{equation*}
so we can write $h_s(\delta z) = \delta \tilde h_{s/\delta^2}(z)$ where $(\tilde h_s)$ is driven by $\delta^{-1} B_{\delta^2 s} =: \tilde B_s$. Note that $\tilde h_{s/\delta^2}'(z) = h_s'(\delta z)$. As before, we denote $z_s = h_s(\delta z)-B_s$ and $\tilde z_s = \tilde h_s(z)-\tilde B_s$, where $z_s = \delta \tilde z_{s/\delta^2}$. Consequently,
\begin{align*}
    & \int_0^t |B_s| \frac{1}{|z_s|^2} \frac{1}{y_s^{1/2}} |h_{s,t}'(z_s)|^{1/2} \, ds \\
    &\quad = \int_0^t |\delta \tilde B_{s/\delta^2}| \frac{1}{\delta^2 |\tilde z_{s/\delta^2}|^2} \frac{1}{\delta^{1/2} \tilde y_{s/\delta^2}^{1/2}} |\tilde h_{s/\delta^2,t/\delta^2}'(\tilde z_{s/\delta^2})|^{1/2} \, ds\\
    &\quad = \delta^{-3/2} \int_0^t |\tilde B_{s/\delta^2}| \frac{1}{|\tilde z_{s/\delta^2}|^2} \frac{1}{\tilde y_{s/\delta^2}^{1/2}} |\tilde h_{s/\delta^2,t/\delta^2}'(\tilde z_{s/\delta^2})|^{1/2} \, ds\\
    &\quad = \delta^{1/2} \int_0^{t/\delta^2} |\tilde B_s| \frac{1}{|\tilde z_s|^2} \frac{1}{\tilde y_s^{1/2}} |\tilde h_{s,t/\delta^2}'(\tilde z_s)|^{1/2} \, ds.
\end{align*}

\noindent Again, for notational simplicity we will stop writing the $\tilde{\ }$ from now on.

\noindent
Now, let $z_0 = i$, and (cf. \cite{Law09})
\begin{equation*}
    \sigma(s) = \inf\{ r \mid y_r = e^{ar}\} = \int_0^s |z_{\sigma(r)}|^2 \, dr
\end{equation*}
which is random and strictly increasing in $s$.

Then
\begin{multline*}
    \delta^{1/2} \int_0^{t/\delta^2} |B_s| \frac{1}{|z_s|^2} \frac{1}{y_s^{1/2}} |h_{s,t/\delta^2}'(z_s)|^{1/2} \, ds \\
    = \delta^{1/2} \int_0^{\sigma^{-1}(t/\delta^2)} |B_{\sigma(s)}| \frac{1}{y_{\sigma(s)}^{1/2}} |h_{{\sigma(s)},t/\delta^2}'(z_{\sigma(s)})|^{1/2} \, ds.
\end{multline*}
This is the integral we will work with.

To sum it up, we have the following.
\begin{proposition}\label{thm:h_integral_reparametrisation}
    Let $z \in \HH$, and $(h_s(\delta z))_{s \ge 0}$ satisfy \eqref{eq:loewner_reverse_2} with $V(s) = \sqrt{\kappa} B_s$ and a standard Brownian motion $B$, and $(\tilde h_s(z))_{s \ge 0}$ satisfy \eqref{eq:loewner_reverse_a} with a standard Brownian motion $\tilde B$. Let $x_s + iy_s = z_s = h_s(\delta z)-V(s)$, and $\tilde x_s + i\tilde y_s = \tilde z_s = \tilde h_s(z) - \tilde B_s$. Then, with the notations above,
    \begin{equation*}
        \int_0^t |B_s| \frac{1}{|z_s|^2} \frac{1}{y_s^{1/2}} |h_{s,t}'(z_s)|^{1/2} \, ds
    \end{equation*}
    has the same law as
    \begin{equation*}
        \kappa^{-3/2} \delta^{1/2} \int_0^{\sigma^{-1}(\kappa t/\delta^2)} |\tilde B_{\sigma(s)}| \frac{1}{\tilde y_{\sigma(s)}^{1/2}} |\tilde h_{{\sigma(s)},\kappa t/\delta^2}'(\tilde z_{\sigma(s)})|^{1/2} \, ds.
    \end{equation*}
    (Recall that $\tilde y_{\sigma(s)}= e^{as}$.)
\end{proposition}

\subsection{Main proof}

In the following, we fix $\kappa \in [\kappa_-,\kappa_+]$, $a = \dfrac{2}{\kappa}$, and let $(h_s(x+i))_{s \ge 0}$ satisfy \eqref{eq:loewner_reverse_a} with initial point $z_0 = x+i$, $|x| \le 1$.

Our goal is to estimate
\begin{multline*}
    \ex \left| \delta^{1/2} \int_0^{\sigma^{-1}(t/\delta^2)} |B_{\sigma(s)}| \frac{1}{y_{\sigma(s)}^{1/2}} |h_{{\sigma(s)},t/\delta^2}'(z_{\sigma(s)})|^{1/2} \, ds \right|^p\\
    = \ex \left| \delta^{1/2} \int_0^\infty 1_{\sigma(s) \le t/\delta^2} |B_{\sigma(s)}| \frac{1}{y_{\sigma(s)}^{1/2}} |h_{{\sigma(s)},t/\delta^2}'(z_{\sigma(s)})|^{1/2} \, ds \right|^p.
\end{multline*}
With \eqref{eq:h_diffkappa_moment} and \cref{thm:h_integral_reparametrisation} this will complete the proof of \cref{thm:f_diffkappa_moment}.

From the definition of $\sigma$ it follows that $\sigma(s) \ge \int_0^s e^{2ar} \, dr = \frac{1}{2a}(e^{2as}-1)$, or equivalently, $\sigma^{-1}(t) \le \frac{1}{2a}\log(1+2at)$. Therefore, $\sigma^{-1}(t/\delta^2) \le \frac{1}{a}\log\frac{C}{\delta}$ and
\begin{multline}\label{eq:integral_moment}
    \ex \left| \delta^{1/2} \int_0^{\sigma^{-1}(t/\delta^2)} |B_{\sigma(s)}| \frac{1}{y_{\sigma(s)}^{1/2}} |h_{{\sigma(s)},t/\delta^2}'(z_{\sigma(s)})|^{1/2} \, ds \right|^p\\
    \le \delta^{p/2} \left( \int_0^{\frac{1}{a}\log\frac{C}{\delta}} \left( \ex \left[ 1_{\sigma(s) \le t/\delta^2} |B_{\sigma(s)}|^p \frac{1}{y_{\sigma(s)}^{p/2}} |h_{{\sigma(s)},t/\delta^2}'(z_{\sigma(s)})|^{p/2} \right] \right)^{1/p} \, ds \right)^p
\end{multline}
where we have applied Minkowski's inequality to pull the moment inside the integral.

To proceed, we need to know more about the behaviour of the reverse SLE flow, which also incorporates the behaviour of $\sigma$. This has been studied in \cite{Law09}. Their tool was to study the process $J_s$ defined by $\sinh J_s = \frac{x_{\sigma(s)}}{y_{\sigma(s)}} = e^{-as}x_{\sigma(s)}$. By \cite[Lemma 6.1]{Law09}, this process satisfies
\begin{equation*}
    dJ_s = -r_c \tanh J_s \, ds + dW_s
\end{equation*}
where $W_s = \int_0^{\sigma(s)} \frac{1}{|z_r|} \, dB_r$ is a standard Brownian motion and $r_c$ is defined in \eqref{eq:moment_estimate_parameters}.

The following results have been originally stated for an equivalent probability measure $\pr_*$, depending on a parameter $r$, such that
\begin{equation*}
    dJ_s = -q \tanh J_s \, ds + dW^*_s
\end{equation*}
with $q>0$ and a process $W^*$ that is a Brownian motion under $\pr_*$. But setting the parameter $r=0$, we have $\pr_* = \pr$, $q=r_c$, and $W^* = W$. Therefore, under the measure $\pr$, the results apply with $q=r_c$.

Note also that although the results were originally stated for a reverse SLE flow starting at $z_0 = i$, they can be written for flows starting at $z_0 = x+i$ without change of the proof. One just uses \cite[Lemma 7.1 (28)]{Law09} with $\cosh J_0 = \sqrt{1+x^2}$.

Recall that \cite{Law09,JVL11} use the notation $\sinh J_s = \frac{x_{\sigma(s)}}{y_{\sigma(s)}}$ and hence $\cosh^2 J_s = 1+\frac{x_{\sigma(s)}^2}{y_{\sigma(s)}^2}$.

\begin{lemma}[{\cite[Lemma 5.6]{JVL11}}]
    Suppose $z_0 = x+i$.
    There exists a constant $C < \infty$, depending on $\kappa_-$, $\kappa_+$, such that for each $s \ge 0$, $u > 0$ there exists an event $E_{u,s}$ with
    \begin{equation*}
        \pr(E_{s,u}^c) \le C (1+x^2)^{r_c} u^{-2r_c}
    \end{equation*}
    on which
    \begin{equation*}
        \sigma(s) \le u^2 e^{2as} \quad \text{and} \quad 1+\frac{x_{\sigma(s)}^2}{y_{\sigma(s)}^2} \le u^2/4.
    \end{equation*}
\end{lemma}

Fix $s \in [0,t]$. Let
\begin{equation*}
    E_u = \left\{ \sigma(s) \le u^2 e^{2as} \text{ and } 1+\frac{x_{\sigma(s)}^2}{y_{\sigma(s)}^2} \le u^2 \right\}
\end{equation*}
and $A_n$ = $E_{\exp(n)} \setminus E_{\exp(n-1)}$ for $n \ge 1$, and $A_0 = E_1$. Then 
\begin{equation}\label{eq:prob_A_n}
\pr(A_n) \le \pr(E_{\exp(n-1)}^c) \le C (1+x^2)^{r_c} e^{-2r_c n} .
\end{equation}
(The constant $C$ may change from line to line.)

\begin{lemma}[see proof of {\cite[Lemma 5.7]{JVL11}}]\label{thm:coshJ_large}
    Suppose $z_0 = x+i$.
    There exists $C < \infty$, depending on $\kappa_-$, and a global constant $\alpha > 0$, such that for all $s \ge 0$, $u > \sqrt{1+x^2}$, and $k > 2a$ we have
    \begin{equation*}
        \pr \left( \sigma(s) \le u^2 e^{2as} \text{ and } 1+\frac{x_{\sigma(s)}^2}{y_{\sigma(s)}^2} \ge u^2 e^k \right) \le C (1+x^2)^{r_c} u^{-2r_c} e^{-\alpha(k-2a)^2}.
    \end{equation*}
\end{lemma}

We proceed to estimating
\begin{multline}\label{eq:moment_on_An}
    \ex \left[ 1_{A_n} 1_{\sigma(s) \le t/\delta^2} |B_{\sigma(s)}|^p \frac{1}{y_{\sigma(s)}^{p/2}} |h_{{\sigma(s)},t/\delta^2}'(z_{\sigma(s)})|^{p/2} \right]\\
    = \ex \left[ 1_{A_n} 1_{\sigma(s) \le t/\delta^2} |B_{\sigma(s)}|^p \frac{1}{y_{\sigma(s)}^{p/2}} \ex \left[ |h_{{\sigma(s)},t/\delta^2}'(z_{\sigma(s)})|^{p/2} \mid \mathcal F_{\sigma(s)} \right] \right]
\end{multline}
where $\mathcal F$ is the filtration generated by $B$.

Note that $y_{\sigma(s)} = e^{as}$ by the definition of $\sigma$. Moreover, on the set $A_n$, the Brownian motion is easy to handle since by Hölder's inequality
\begin{align}
    \ex [ 1_{A_n} 1_{\sigma(s) \le t/\delta^2} |B_{\sigma(s)}|^p ] &\le \ex \left[ 1_{A_n} 1_{\sigma(s) \le t/\delta^2} \sup_{r \in [0, e^{2n} e^{2as}]} |B_r|^p \right] \nonumber\\
    &\le \pr(A_n \cap \{\sigma(s) \le t/\delta^2\})^{1-\varepsilon}\, \ex \left[ \sup_{r \in [0, e^{2n} e^{2as}]} |B_r|^{p/\varepsilon} \right]^{\varepsilon} \nonumber\\
    &\le C\, \pr(A_n \cap \{\sigma(s) \le t/\delta^2\})^{1-\varepsilon}\, e^{np} e^{pas} \label{eq:bm_moment}
\end{align}
for any $\varepsilon > 0$.

It remains to handle $\ex \left[ |h_{{\sigma(s)},t/\delta^2}'(z_{\sigma(s)})|^{p/2} \mid \mathcal F_{\sigma(s)} \right]$.

The following result is well-known and follows from the Schwarz lemma and mapping the unit disc to the half-plane.
\begin{lemma}
    Let $f: \HH \to \HH$ be a holomorphic function. Then $|f'(z)| \le \frac{\Im(f(z))}{\Im(z)}$ for all $z \in \HH$.
\end{lemma}

Recall that the Loewner equation implies
\[ \Im(h_{{\sigma(s)},t/\delta^2}(z_{\sigma(s)})) = y_{t/\delta^2} \le \sqrt{1+2at/\delta^2} \le C\delta^{-1} .\]
Let $\varepsilon > 0$. By the lemma above, we can estimate
\begin{multline}\label{eq:hprime1}
    \ex \left[ |h_{{\sigma(s)},t/\delta^2}'(z_{\sigma(s)})|^{p/2} \mid \mathcal F_{\sigma(s)} \right] \\
    \le (\delta y_{\sigma(s)})^{-(1-\varepsilon)p/2} \ex \left[ |h_{{\sigma(s)},t/\delta^2}'(z_{\sigma(s)})|^{\varepsilon p/2} \mid \mathcal F_{\sigma(s)} \right].
\end{multline}

From \cite[Lemma 3.2]{JVL11} it follows that there exists some $l>0$ such that
\begin{equation}\label{eq:hprime2}
    |h_{{\sigma(s)},t/\delta^2}'(z_{\sigma(s)})| \le C\left(1+\frac{x_{\sigma(s)}^2}{y_{\sigma(s)}^2}\right)^l |h_{{\sigma(s)},t/\delta^2}'(iy_{\sigma(s)})|.
\end{equation}

We claim that
\begin{equation}\label{eq:hprime3}
    \ex \left[ |h_{{\sigma(s)},t/\delta^2}'(iy_{\sigma(s)})|^{\varepsilon p/2} \mid \mathcal F_{\sigma(s)} \right] \le C
\end{equation}
if $\varepsilon>0$ is sufficiently small.

To see this, first recall that for small $\varepsilon > 0$ we have
\begin{equation}\label{eq:hprime_small_moment}
\ex\left[\abs{h_t'(i)}^\varepsilon\right] \le C
\end{equation}
uniformly in $t \ge 1$. This follows from \cite[Theorem 5.4]{JVL11} or, even more elementary, from the proof of \cite[Theorem 3.2]{RS05}.

Now approximate $\sigma(s)$ by simple stopping times $\tilde \sigma \ge \sigma(s)$. A possible choice is $\tilde \sigma = \lceil \sigma(s) 2^n \rceil 2^{-n} \wedge t/\delta^2$. It suffices to show
\[ \ex \left[ |h_{{\tilde\sigma},t/\delta^2}'(iy_{\sigma(s)})|^{\varepsilon p/2} \mid \mathcal F_{\sigma(s)} \right] \le C \]
and then apply Fatou's lemma to pass to the limit.

Now that $\tilde\sigma$ is simple, we can apply \eqref{eq:hprime_small_moment} on each set $F_r = \{ \tilde \sigma = r \}$. Using the strong Markov property of Brownian motion and the scaling invariance of SLE, we get
\begin{align*}
    \ex \left[ 1_{F_r} |h_{{\tilde \sigma},t/\delta^2}'(ie^{as})|^{\varepsilon p/2} \mid \mathcal F_{\sigma(s)} \right] &= 1_{F_r} \ex \left[ |h_{r,t/\delta^2}'(ie^{as})|^{\varepsilon p/2} \right] \\
    &= 1_{F_r} \ex \left[ |h_{e^{-2as}(t/\delta^2-r)}'(i)|^{\varepsilon p/2} \right]\\
    &\le 1_{F_r} C
\end{align*}
and the claim follows.

Combining \eqref{eq:hprime1}, \eqref{eq:hprime2}, and \eqref{eq:hprime3}, we have
\begin{align}
    \ex \left[ |h_{{\sigma(s)},t/\delta^2}'(z_{\sigma(s)})|^{p/2} \mid \mathcal F_{\sigma(s)} \right]
    &\le C\, \delta^{-(1-\varepsilon)p/2}\, y_{\sigma(s)}^{-(1-\varepsilon)p/2} \left(1+\frac{x_{\sigma(s)}^2}{y_{\sigma(s)}^2}\right)^{l\varepsilon p/2} \nonumber\\
    &\le C\, \delta^{-(1-\varepsilon)p/2}\, e^{-(1-\varepsilon)pas/2} \left(1+\frac{x_{\sigma(s)}^2}{y_{\sigma(s)}^2}\right)^{l\varepsilon p/2} \label{eq:hprime_final}
\end{align}
where on the set $A_n$ we have
\begin{equation*}
    1+\frac{x_{\sigma(s)}^2}{y_{\sigma(s)}^2} \le e^{2n}.
\end{equation*}

Proceeding from \eqref{eq:moment_on_An}, we get from \eqref{eq:hprime_final} and \eqref{eq:bm_moment}
\begin{align}
    &\ex \left[ 1_{A_n} 1_{\sigma(s) \le t/\delta^2} |B_{\sigma(s)}|^p \frac{1}{y_{\sigma(s)}^{p/2}} \ex \left[ |h_{{\sigma(s)},t/\delta^2}'(z_{\sigma(s)})|^{p/2} \mid \mathcal F_{\sigma(s)} \right] \right] \nonumber\\
    &\quad \le C\, \ex \left[ 1_{A_n} 1_{\sigma(s) \le t/\delta^2}\, |B_{\sigma(s)}|^p\, e^{-pas/2}\, \delta^{-(1-\varepsilon)p/2}\, e^{-(1-\varepsilon)pas/2} e^{nl\varepsilon p} \right] \nonumber\\
    &\quad \le C\, \delta^{-(1-\varepsilon)p/2}\, e^{nl\varepsilon p}\, e^{-pas+\varepsilon pas/2}\, \pr(A_n \cap \{\sigma(s) \le t/\delta^2\})^{1-\varepsilon}\, e^{np} e^{pas} \nonumber\\
    &\quad = C\, \delta^{-(1-\varepsilon)p/2}\, e^{np+nl\varepsilon p}\, e^{\varepsilon pas/2}\, \pr(A_n \cap \{\sigma(s) \le t/\delta^2\})^{1-\varepsilon}. \label{eq:moment_on_An_result}
\end{align}

We would like to sum this expression in $n$.
\begin{proposition}\label{thm:sum_An}
Let $\sigma(s)$ and $A_n$ be defined as above. Then
\begin{multline*}
    \sum_{n \in \NN} e^{np+nl\varepsilon p}\, \pr(A_n \cap \{\sigma(s) \le t/\delta^2\})^{1-\varepsilon} \\
    \le \begin{cases}
    C & \text{if } p+l\varepsilon p-2r_c(1-\varepsilon) < 0\\
    C(e^{-as}\sqrt{t}/\delta)^{p+l\varepsilon p-2r_c(1-\varepsilon)} & \text{if } p+l\varepsilon p-2r_c(1-\varepsilon) > 0
    \end{cases}
\end{multline*}
where $C < \infty$ depends on $\kappa_-$, $\kappa_+$, $p$, and $\varepsilon$.
\end{proposition}

\begin{proof}
We distinguish two cases. If $n \le {\log(\sqrt{t}/\delta)-as+1+a}$, we have (by \eqref{eq:prob_A_n})
\begin{align*}
    & \sum_{n \le \log(\sqrt{t}/\delta)-as+1+a} e^{np+nl\varepsilon p}\, \pr(A_n)^{1-\varepsilon} \nonumber\\
    &\quad \le C \sum_{n \le \log(\sqrt{t}/\delta)-as+1+a} e^{np+nl\varepsilon p} e^{-2nr_c(1-\varepsilon)} \nonumber\\
    &\quad \le \begin{cases}
    C & \text{if } p+l\varepsilon p-2r_c(1-\varepsilon) < 0\\
    C(e^{-as}\sqrt{t}/\delta)^{p+l\varepsilon p-2r_c(1-\varepsilon)} & \text{if } p+l\varepsilon p-2r_c(1-\varepsilon) > 0.
    \end{cases}
\end{align*}

For $n > {\log(\sqrt{t}/\delta)-as+1+a}$, we have $e^{2(n-1)}e^{2as} > t/\delta^2$ and therefore (by the definition of $A_n$)
\begin{align*}
    A_n \cap \{\sigma(s) \le t/\delta^2\} &\subseteq E_{e^{n-1}}^c \cap \{\sigma(s) \le t/\delta^2\} \\
    &\subseteq \left\{\sigma(s) \le t/\delta^2 \text{ and } 1+\frac{x_{\sigma(s)}^2}{y_{\sigma(s)}^2} > e^{2(n-1)} \right\},
\end{align*}
so \cref{thm:coshJ_large}, applied to $u=e^{-as}\sqrt{t}/\delta$ and $k=2(n-1)-2(\log(\sqrt{t}/\delta)-as)$, implies
\begin{align*}
    \pr(A_n \cap \{\sigma(s) \le t/\delta^2\}) &\le C\, (e^{-as}\sqrt{t}/\delta)^{-2r_c}\, e^{-\alpha(2(n-1)-2(\log(\sqrt{t}/\delta)-as)-2a)^2}\\
    &= C\, (e^{-as}\sqrt{t}/\delta)^{-2r_c}\, e^{-2\alpha(n-(\log(\sqrt{t}/\delta)-as+1+a))^2}.
\end{align*}
Consequently,
\begin{align*}
    &\sum_{n > \log(\sqrt{t}/\delta)-as+1+a} e^{np+nl\varepsilon p}\, \pr( A_n \cap \{\sigma(s) \le t/\delta^2\} )^{1-\varepsilon} \nonumber\\
    &\quad \le C (e^{-as}\sqrt{t}/\delta)^{p+l\varepsilon p} \sum_{n \in \NN} e^{np+nl\varepsilon p}\, (e^{-as}\sqrt{t}/\delta)^{-2r_c(1-\varepsilon)}\, e^{-2\alpha(1-\varepsilon) n^2} \nonumber\\
    &\quad \le C (e^{-as}\sqrt{t}/\delta)^{p+l\varepsilon p-2r_c(1-\varepsilon)}.
\end{align*}
\end{proof}

Hence, by \eqref{eq:moment_on_An_result} and \cref{thm:sum_An},
\begin{align}
    &\ex \left[ 1_{\sigma(s) \le t/\delta^2} |B_{\sigma(s)}|^p \frac{1}{y_{\sigma(s)}^{p/2}} |h_{{\sigma(s)},t/\delta^2}'(z_{\sigma(s)})|^{p/2} \right] \nonumber\\
    &\quad = \sum_{n=0}^\infty \ex \left[ 1_{A_n} 1_{\sigma(s) \le t/\delta^2} |B_{\sigma(s)}|^p \frac{1}{y_{\sigma(s)}^{p/2}} |h_{{\sigma(s)},t/\delta^2}'(z_{\sigma(s)})|^{p/2} \right] \nonumber\\
    &\quad \le \begin{cases}
    C\, \delta^{-(1-\varepsilon)p/2}\, e^{\varepsilon pas/2} & \text{if } p+l\varepsilon p-2r_c(1-\varepsilon) < 0\\
    C\, \delta^{-(1-\varepsilon)p/2}\, (e^{-as}\sqrt{t}/\delta)^{p+l\varepsilon p-2r_c(1-\varepsilon)}\, e^{\varepsilon pas/2} & \text{if } p+l\varepsilon p-2r_c(1-\varepsilon) > 0.
    \end{cases}\label{eq:integrand_moment}
\end{align}

Finally, if $p+l\varepsilon p-2r_c(1-\varepsilon) < 0$, we estimate \eqref{eq:integral_moment} with \eqref{eq:integrand_moment}, so
\begin{align*}
    &\ex \left| \delta^{1/2} \int_0^{\sigma^{-1}(t/\delta^2)} |B_{\sigma(s)}| \frac{1}{y_{\sigma(s)}^{1/2}} |h_{{\sigma(s)},t/\delta^2}'(z_{\sigma(s)})|^{1/2} \, ds \right|^p\\
    &\quad \le \delta^{p/2} \left( \int_0^{\frac{1}{a}\log\frac{C}{\delta}} \left( \ex \left[ 1_{\sigma(s) \le t/\delta^2} |B_{\sigma(s)}|^p \frac{1}{y_{\sigma(s)}^{p/2}} |h_{{\sigma(s)},t/\delta^2}'(z_{\sigma(s)})|^{p/2} \right] \right)^{1/p} \, ds \right)^p\\
    &\quad \le C \delta^{p/2} \left( \int_0^{\frac{1}{a}\log\frac{C}{\delta}} \left( \delta^{-(1-\varepsilon)p/2}\, e^{\varepsilon pas/2} \right)^{1/p} \, ds \right)^p\\
    &\quad = C \delta^{\varepsilon p/2} \left( \int_0^{\frac{1}{a}\log\frac{C}{\delta}} e^{\varepsilon as/2} \, ds \right)^p\\
    &\quad \le C.
\end{align*}

Since $\varepsilon > 0$ can be chosen as small as we want, the condition to apply this is $p < 2r_c = 1+\frac{8}{\kappa}$.

On the other hand, if $p+l\varepsilon p-2r_c(1-\varepsilon) > 0$, we have
\begin{align*}
    &\ex \left| \delta^{1/2} \int_0^{\sigma^{-1}(t/\delta^2)} |B_{\sigma(s)}| \frac{1}{y_{\sigma(s)}^{1/2}} |h_{{\sigma(s)},t/\delta^2}'(z_{\sigma(s)})|^{1/2} \, ds \right|^p\\
    &\quad \le C \delta^{p/2} \left( \int_0^{\frac{1}{a}\log\frac{C}{\delta}} \left( \delta^{-(1-\varepsilon)p/2}\, (e^{-as}\sqrt{t}/\delta)^{p+l\varepsilon p-2r_c(1-\varepsilon)}\, e^{\varepsilon pas/2} \right)^{1/p} \, ds \right)^p\\
    &\quad \le C \delta^{\varepsilon p/2-(p+l\varepsilon p-2r_c(1-\varepsilon))} \left( \int_0^{\frac{1}{a}\log\frac{C}{\delta}} e^{as(\varepsilon/2-(1+l\varepsilon -2r_c(1-\varepsilon)/p))} \, ds \right)^p\\
    &\quad \le \begin{cases}
    C & \text{if } \varepsilon/2-(1+l\varepsilon -2r_c(1-\varepsilon)/p) > 0\\
    C \delta^{\varepsilon p/2-(p+l\varepsilon p-2r_c(1-\varepsilon))} & \text{if } \varepsilon/2-(1+l\varepsilon -2r_c(1-\varepsilon)/p) < 0
    \end{cases}\\
    &\quad =\begin{cases}
    C & \text{if } 2r_c(1-\varepsilon) - p(1+\varepsilon(l-1/2)) > 0\\
    C \delta^{2r_c(1-\varepsilon) - p(1+\varepsilon(l-1/2))} & \text{if } 2r_c(1-\varepsilon) - p(1+\varepsilon(l-1/2)) < 0.
    \end{cases}
\end{align*}

Since $\varepsilon > 0$ can be chosen as small as we want, the condition to apply this is $p > 2r_c = 1+\frac{8}{\kappa}$, and the exponent can be chosen to be greater than $2r_c-p-\varepsilon'$ for any $\varepsilon' > 0$.

With this estimate for \eqref{eq:integral_moment}, the proof of \cref{thm:f_diffkappa_moment} is complete.


\appendix

\section{Appendix: Proof of \cref{thm:grr_integrals_expectation}}
\label{sec:grr_integrals_expectation_pf}

We begin with estimating the expressions for $A_{1j}$ which involve the time difference, and then estimate the expressions for $A_{2j}$ which involve the $\kappa$ difference.

\subsubsection*{The $\Delta t$ term}

For this part, we again suppress writing $\kappa$, although all expressions depend on a parameter $\kappa$.

The moment estimates are all similar. In $A_{11}$, we will encounter the expression $\ex |\hat f_t'(iu)-\hat f_s'(iu)|^\lambda$ which we estimate by $\ex |\hat f_t'(iu)-\hat f_s'(iu)|^\lambda \lesssim (a(s)+a(t)) u^\zeta$ with \cref{thm:moment_estimate} (which is sufficient since $|t-s| \ge u^2$). Together with Minkowski's inequality, we have
\begin{align*}
    \ex |A_{11}(t,s;\kappa)|^\lambda &\le \left( \int_0^{y \wedge |t-s|^{1/2}} \left( \ex|\hat f_t'(iu)-\hat f_s'(iu)|^\lambda \right)^{1/\lambda} \, du \right)^\lambda \\
    &\lesssim \left( \int_0^{y \wedge |t-s|^{1/2}} (a(s)+a(t))^{1/\lambda} u^{\zeta/\lambda} \, du \right)^\lambda \\
    &\lesssim (a(s)+a(t)) \left( y \wedge |t-s|^{1/2} \right)^{\zeta+\lambda},
\end{align*}
assuming $\zeta+\lambda > 0$. Consequently,
\begin{align*}
\ex \iint \frac{|A_{11}(t,s;\kappa)|^\lambda}{|t-s|^{\beta_1}} \, ds \, dt 
&\lesssim \iint_{|t-s| \le y^2} \frac{(a(s)+a(t)) |t-s|^{(\zeta+\lambda)/2}}{|t-s|^{\beta_1}} \, ds \, dt \\
&\qquad + \iint_{|t-s| > y^2} \frac{(a(s)+a(t)) y^{\zeta+\lambda}}{|t-s|^{\beta_1}} \, ds \, dt \\
&\lesssim y^{\zeta+\lambda-2\beta_1+2} \int a(t) \, dt,
\end{align*}
assuming $1 < \beta_1 < \frac{\zeta+\lambda}{2}+1$.

The terms $A_{12}$, $A_{13}$ only appear when $|t-s|^{1/2} < y$.

For $A_{12}$, we get (again by Minkowski's inequality and \cref{thm:moment_estimate})
\begin{align*}
    \ex |A_{12}(t,s;\kappa)|^\lambda &\le \left( \int_{y \wedge |t-s|^{1/2}}^y \frac{|t-s|}{u^2} \left( \ex|\hat f_s'(iu)|^\lambda \right)^{1/\lambda} \, du \right)^\lambda \\
    &\le |t-s|^{\lambda} \left( \int_{y \wedge |t-s|^{1/2}}^y a(s)^{1/\lambda} u^{\zeta/\lambda} u^{-2} \, du \right)^\lambda \\
    &\lesssim a(s) |t-s|^{(\zeta+\lambda)/2}
\end{align*}
using the fact that $\zeta < \lambda$ (see \eqref{eq:moment_estimate_parameters}).

Finally, for $A_{13}$, note that $\norm{B}_{C^{1/2^{(-)}}}$ has arbitrarily high moments, so that we can apply Hölder's inequality and get
\[ 
\ex\left[ \left( |\hat f_s'(iu)| \left(1+\norm{B}_{C^{1/2^{(-)}}}\right)^{2l+1} \right)^{\lambda^{(-)}} \right]
\lesssim \left( \ex |\hat f_s'(iu)|^\lambda \right)^{1^{(-)}}
\lesssim (a(s) u^\zeta)^{1^{(-)}}
\]
Consequently, again by Minkowski's inequality,
\begin{align*}
    \ex |A_{13}(t,s;\kappa)|^{\lambda^{(-)}} &\lesssim \left( \int_{y \wedge |t-s|^{1/2}}^{y \wedge 2|t-s|^{1/2}} u^{-1} (a(s) u^\zeta)^{1/\lambda^{(-)}} |t-s|^{1/2^{(-)}} \, du \right)^{\lambda^{(-)}}\\
    &\lesssim a(s)^{1^{(-)}} |t-s|^{(\zeta+\lambda)^{(-)}/2}.
\end{align*}

This shows
\begin{align*}
\ex \iint \frac{|A_{12}(t,s;\kappa)|^\lambda}{|t-s|^{\beta_1}} \, ds \, dt &\lesssim y^{\zeta+\lambda-2\beta_1+2} \int a(t) \, dt, \\
\ex \iint \frac{|A_{13}(t,s;\kappa)|^{\lambda^{(-)}}}{|t-s|^{\beta_1}} \, ds \, dt &\lesssim y^{(\zeta+\lambda)^{(-)} -2\beta_1+2} \int a(t)^{1^{(-)}} \, dt
\end{align*}
if $\beta_1 < \frac{\zeta+\lambda}{2}+1$.

\subsubsection*{The $\Delta \kappa$ term}

$A_{21}$ will again just be estimated using \cref{thm:moment_estimate} on
\begin{equation*}
\ex |(\hat f^\kappa_s)'(iu)-(\hat f^{\tilde\kappa}_s)'(iu)|^\lambda \lesssim a(s,\zeta_\kappa) u^{\zeta_\kappa} + a(s,\zeta_{\tilde\kappa}) u^{\zeta_{\tilde\kappa}}.
\end{equation*}

Then, by Minkowski's inequality,
\begin{align*}
\ex |A_{21}(s;\kappa,\tilde\kappa)|^\lambda 
&\le \left( \int_0^{y \wedge |\kappa-\tilde\kappa|^{p/(\zeta+\lambda)}} \left( \ex|(\hat f^\kappa_s)'(iu)-(\hat f^{\tilde\kappa}_s)'(iu)|^\lambda \right)^{1/\lambda} \, du \right)^\lambda\\
&\lesssim \left( \int_0^{y \wedge |\kappa-\tilde\kappa|^{p/(\zeta+\lambda)}} a(s,\zeta_\kappa)^{1/\lambda} u^{\zeta/\lambda} \, du \right)^\lambda\\
&\lesssim a(s,\zeta_\kappa) (y^{\zeta+\lambda} \wedge |\kappa-\tilde\kappa|^p),
\end{align*}
assuming $\zeta+\lambda > 0$, and consequently
\begin{align*}
\ex \iint \frac{|A_{21}(s;\kappa,\tilde\kappa)|^{\lambda}}{|\kappa-\tilde\kappa|^{\beta_2}} \, d\kappa \, d\tilde\kappa \lesssim y^{(\zeta+\lambda)(p-\beta_2+1)/p} \int a(s,\zeta_\kappa) \, d\kappa,
\end{align*}
assuming $p-\beta_2+1 > 0$, i.e. $\beta_2 < p+1$.

For $A_{22}$ we apply \cref{thm:fprime_diffkappa_moment} when $\kappa$, $\tilde\kappa$ are close to each other, i.e. $|\kappa-\tilde\kappa|^{p/(\zeta+\lambda)} \le y$. This gives us
\begin{equation*}
\ex |(\hat f^\kappa_s)'(iu)-(\hat f^{\tilde\kappa}_s)'(iu)|^p \lesssim |\kappa-\tilde\kappa|^p u^{-p}.
\end{equation*}

In this case Minkowski's inequality does not give us quite the optimal estimate (although it is still sufficient), therefore we do something similar. Let $b \in \RR$ be a constant that will be chosen later.

By H\"older's inequality,
\begin{align*}
\ex |A_{22}(s;\kappa,\tilde\kappa)|^p 
&= \ex \left( \int_{|\kappa-\tilde\kappa|^{p/(\zeta+\lambda)}}^y |(\hat f^\kappa_s)'(iu)-(\hat f^{\tilde\kappa}_s)'(iu)|\, u^b u^{-b} \, du \right)^p\\
&\lesssim \ex \left( \int_{|\kappa-\tilde\kappa|^{p/(\zeta+\lambda)}}^y |(\hat f^\kappa_s)'(iu)-(\hat f^{\tilde\kappa}_s)'(iu)|^p u^{bp} \, du \right) y^{p-1-bp}\\
&\lesssim \left( \int_{|\kappa-\tilde\kappa|^{p/(\zeta+\lambda)}}^y |\kappa-\tilde\kappa|^p u^{-p} u^{bp} \, du \right) y^{p-1-bp}\\
&\lesssim |\kappa-\tilde\kappa|^{p+(-p+bp+1)p/(\zeta+\lambda)} y^{p-1-bp},
\end{align*}
assuming $p-1-bp > 0$.

Then
\begin{align*}
\ex \iint \frac{|A_{22}(s;\kappa,\tilde\kappa)|^p}{|\kappa-\tilde\kappa|^{\beta_2}} \, d\kappa \, d\tilde\kappa \lesssim y^{(\zeta+\lambda)(p-\beta_2+1)/p},
\end{align*}
assuming $p+(-p+bp+1)p/(\zeta+\lambda)-\beta_2+1 > 0$.

These estimates work if we can choose $b$ such that\\
$bp \in {] (\beta_2-p-1)(\zeta+\lambda)/p+p-1,\, p-1[}$, i.e. when $\beta_2 < p+1$.

This finishes the estimates of \cref{thm:grr_integrals_expectation}.

\bibliographystyle{alpha}

\end{document}